\documentclass[11pt]{amsart}
\usepackage{fullpage}
\usepackage[utf8]{inputenc}
\usepackage[T1]{fontenc}
\usepackage{graphicx}
\usepackage{amsmath,amsfonts,amssymb,mathtools}
\usepackage{stmaryrd}
\usepackage{graphicx}
\usepackage{mathrsfs,dsfont}
\usepackage{amsmath,amsthm,amsthm}
\usepackage[active]{srcltx}
\usepackage{comment}

\usepackage{hyperref}

\usepackage{enumerate}

\usepackage{color}

\newcommand{\E}{\mathbb{E}}
\newcommand{\R}{\mathbb{R}}
\newcommand{\N}{\mathbb{N}}
\newcommand{\PP}{\mathbb{P}}

\newcommand{\cF}{\mathcal{F}} 

\numberwithin{equation}{section}

\newtheorem{theo}{Theorem}[section]

\newtheorem{cor}[theo]{Corollary}
\newtheorem{rem}[theo]{Remark}

\newtheorem{propo}[theo]{Proposition}
\newtheorem{lemma}[theo]{Lemma}

\newcommand{\sgc}{\begin{color}{red}}
\newcommand{\cgs}{\end{color}}

\newcommand{\eps}{\varepsilon}
\newcommand{\Tr}{\operatorname{Tr}}

\title{Weak error estimates of Galerkin approximations for the stochastic Burgers equation driven by additive trace-class noise}

\author{Charles-Edouard Bréhier}
\address{Universite de Pau et des Pays de l'Adour, E2S UPPA, CNRS, LMAP, Pau, France}
\email{charles-edouard.brehier@univ-pau.fr}
\author{Sonja Cox}
\address{Korteweg-de Vries Institute for Mathematics,
         University of Amsterdam,
         Postbus 94248,
         NL-1090 GE Amsterdam, Netherlands}
\email{s.g.cox@uva.nl}
\author{Annie Millet}
\address{SAMM, Universit\'e Paris 1 Panth\'on Sorbonne,
90 Rue de Tolbiac, 75013 PARIS, France \& LPSM, UMR 8001 }
\email{annie.millet@univ-paris1.fr}


\begin{document}

\begin{abstract}
We establish weak convergence rates for spectral Galerkin approximations of the stochastic  viscous Burgers equation driven by additive trace-class noise. Our results complement the known results regarding strong convergence; we obtain essential weak convergence rate 2. As expected, this is twice the known strong rate. The main ingredients of the proof are novel regularity results on the solutions of the associated Kolmogorov equations.
\end{abstract}

\maketitle

\section{Introduction}

Stochastic partial differential equations (SPDEs) have become a popular tool in modelling many complex phenomena. 
The design and analysis of numerical methods for SPDEs has thus become an important field of research. This is a challenging task which has been performed by many authors in the last decades.
A general theory has been developed for abstract and general semilinear SPDEs of parabolic type, when the nonlinearities are assumed to be globally Lipschitz continuous. 
However, for more complicated models  depending on nonlinearities with polynomial growth, such as the stochastic Burgers and Navier--Stokes equations, and the stochastic Allen--Cahn or Cahn--Hilliard equations, 
more care is needed in the construction and in the analysis of the schemes. Many questions regarding the rates of convergence of temporal and spatial discretization schemes for these types of problems remain open. 
The objective of this work is to fill a gap in the literature: we establish weak convergence rates for a spatial discretization of the 
stochastic Burgers equation driven by additive trace-class noise, performed using the spectral Galerkin method.

We consider the one-dimensional stochastic  viscous Burgers equation with homogeneous Dirichlet boundary conditions and additive trace class noise, i.e., the equation formally given by:
\begin{equation}\label{eq:Burgers_intro}
\left\lbrace
\begin{aligned}
&\partial_tX(t,z)=\Delta X(t,z)+\nabla(X(t,z)^2)+\dot{W}^Q(t,z),\quad t>0, z\in(0,1),\\
&X(t,0)=X(t,1)=0,\quad t>0,\\
&X(0,z)=X_0(z),\quad z\in(0,1),
\end{aligned}
\right.
\end{equation}
where $\nabla$ and $\Delta$ denote the first and second order derivatives with respect to $z$, and where the random initial value $X_0$ 
is assumed to be given. The evolution is driven by Gaussian noise which is white in time and colored in space; 
more precisely, we let $\bigl(W^Q(t)\bigr)_{t\ge 0}$ be a $L^2(0,1)$-valued Wiener process given by
\begin{equation}  \label{eq:defWQ_intro}
W^Q(t) = \sum_{k\in\N}\sqrt{q_k}W^{(k)}(t)e_k,\quad \forall~t\in\R^+,
\end{equation}
where $\bigl(e_k\bigr)_{k\in\N}$ is a complete orthonormal system of the Hilbert space $L^2(0,1)$, $\bigl(W^{(k)}( \cdot)\bigr)_{k\in\N}$ is a 
sequence of independent standard real-valued Wiener processes independent of $X_0$, and $\bigl(q_k\bigr)_{k\in\N}$ is a sequence of nonnegative real numbers such that $\sum_{k\in\N}q_k<\infty.$

For rigorous statements and analysis, it is convenient to interpret the stochastic partial differential equation~\eqref{eq:Burgers_intro} as a stochastic evolution equation  in the framework developed in~\cite{DaPratoZabczyk:1992},  i.e., we now consider
\begin{equation}\label{eq:SPDE}
\left\lbrace
\begin{aligned}
&dX(t)=AX(t)\,dt+B(X(t))\,dt+\,dW^Q(t),\quad \forall t\ge 0,\\
&X(0)=X_0,
\end{aligned}
\right.
\end{equation}
where the unknown $\bigl(X(t)\bigr)_{t\ge 0}$ is a $L^2(0,1)$-valued stochastic process. Here $A$ is the Dirichlet Laplace operator and  $B(X)=\nabla(X^2)$, for details see Section \ref{ssec:setting}.
 Well-posedness, moment and exponential moment bounds on the solutions are standard results;  they are recalled in Sections~\ref{sec:preliminaries} and ~\ref{sec:bounds} for the readers' convenience.

In this work, we consider the spectral Galerkin approximation
\begin{equation}\label{eq:Galerkin-Burgers}
\left\lbrace
\begin{aligned}
&dX_{M}(t)=AX_{M}(t)\,dt+P_{M} B(X_{M}(t))\,dt+P_{M}\,dW^Q(t),\quad \forall ~ t\ge 0,\\
&X_{M}(0)=P_{M} X_0,
\end{aligned}
\right.
\end{equation}
where $P_M$ ($M\in \{1,2,\ldots \}$) denotes the orthogonal projection operator in $L^2(0,1)$ onto the span of the first $M$ eigenfunctions of $A$.
The strong convergence rate of $X_M$ against $X$ was established in~\cite{HutzenthalerJentzen:2020}; it follows from~\cite[Section 3.2.3]{HutzenthalerJentzen:2020} 
that if $X_0\equiv x_0\in W^{1,2}$, then for all $p\in [1,\infty)$ and all $\alpha \in [0,1)$ one has
\begin{equation}\label{eq:strong_conv_Galerkin}
    \underset{M\in \N,\, t\in [0,T]}\sup~M^{\alpha}\E\big[ \| X(t) - X_{M}(t)\|_{L^2}^p\big]^{\frac{1}{p}}
    < \infty.
\end{equation} 
 The above means that the spectral Galerkin approximation of the stochastic Burgers equation driven by additive trace-class noise 
 converges with essential strong rate $1$.

The objective of this work is to establish weak error estimates, which are concerned with the rate of convergence of
\[
\big|\E[\varphi(X(T))]-\E[\varphi(X_{M}(T))]\big|
\]
where $\varphi:L^2(0,1)\to\R$ is  an arbitrary sufficiently smooth  function, i.e., twice continuously differentiable with bounded first and 
second order derivatives. Since such functions are globally Lipschitz continuous, the rate of convergence of the weak error is at least equal to 
the strong order. However, in many situations, as will be reviewed below, the weak rate of convergence is twice the strong rate.\par 

The primary goal of this work is to prove that the weak convergence rate for the spectral Galerkin approximation~\eqref{eq:Galerkin-Burgers} 
 applied to the stochastic Burgers equation~\eqref{eq:SPDE} is indeed twice the strong rate, which, to the best of our knowledge, is a new result. 
 More specifically, we prove that if there exist $\gamma_0\in (0,\infty)$ and $p\in (32,\infty)$ such that $\E [\exp(\gamma_0 \| X_0\|_{L^2}^2 ) ]< \infty$ 
 and $\E\bigl[\| \nabla X_0 \|_{L^2}^{p}\bigr]<\infty$ (clearly this holds if $X_0\equiv x_0\in W^{1,2}$  is non-random),
 and if $\varphi\in C^2(H,\R)$ has bounded first and second derivatives, then for all $\alpha \in [0,1)$ one has
\begin{equation}
    \sup_{M\in \N} M^{2\alpha} \big|\E[ \varphi(X(T))] - \E[\varphi(X_M(T))]\big|
    < \infty.
\end{equation}
See Theorem \ref{thm:weakrate} for a precise and slightly more general statement. 
Consequently, the spectral Galerkin approximation of the stochastic Burgers equation driven by additive trace-class noise converges with 
essential weak rate $2$, which is twice the strong rate. 
There are strong indications that the strong rate 1 and weak rate 2 are optimal, for details see Remark~\ref{rem:notoptimal}. 

Aside from aforementioned work~\cite{HutzenthalerJentzen:2020}, strong convergence rates for approximations of the Burgers' equation~\eqref{eq:Burgers_intro} can also be found in~\cite{HutzenthalerEtAl:2019}, where space-time discretizations are considered. Strong convergence rates (with rate $1/2$ in time and rate $1$ in space, for an error measured in $L^p(\Omega,L^2)$) of a fully discrete scheme for the stochastic Burgers equation driven by multiplicative noise can be found in the recent work~\cite{HutzenthalerLink:2022}. We refer for instance to~\cite{MR4454935,MR2860933} for further results on approximations for the stochastic Burgers equation.
Let us also mention that convergence results for the stochastic Burgers equation driven by space-time white noise (which is not covered in this work) have been obtained, for instance in~\cite{AlabertGyongy:2006,BlomkerJentzen:2013,JentzenSalimovaWelti:2019,MazzonettoSalimova:2020}. 

In order to obtain weak convergence rates, we need to establish appropriate regularity properties for solutions to the Kolmogorov equations associated with the stochastic evolution equation~\eqref{eq:SPDE}. 
We refer to Theorem~\ref{theo:Kol_reg} for a precise statement of this result, which is one of the main contributions of this work and may be useful in other contexts. 
Instead of considering infinite-dimensional problems, it is convenient to consider the Galerkin approximation~\eqref{eq:Galerkin-Burgers} with arbitrary $M$, 
however, notice that the upper bounds below are independent of dimension $M$. Set $u_M(t,x)=\E[\varphi(X_M^x(t)]$, where $X_M^x$ denotes the solution of~\eqref{eq:Galerkin-Burgers} with initial value $X_M(0)=P_Mx$.
The mapping $u_M$ is then a solution to the Kolmogorov equation, see~\eqref{eq:Kolmogorov_eqn}. 
Theorem~\ref{theo:Kol_reg} provides upper bounds for the first and second order derivatives of $u_M(t,.)$, which can be described as follows. 
For all $\alpha\in[0,1)$, one has upper bounds of the following type for the first order derivative:
\[
|Du_M(t,x).h| \leq C_{\alpha,\delta,\epsilon}(T,Q,\varphi) t^{-\alpha} e^{\epsilon \|x\|_{L^2}^2} \bigl(1+\|(-A)^{\frac{1}{4}+\delta}x\|_{L^2}^6\bigr)
\|(-A)^{-\alpha} h\|_{L^2}.
\]
Moreover, for all $\alpha,\beta\in[0,1)$ such that $\alpha+\beta<1$, one has upper bounds of the following type for the second order derivative:
\[
|D^2u_M(t,x).(g,h)| \leq C_{\alpha,\beta,\delta,\epsilon}(T,Q,\varphi) t^{-(\alpha+\beta)} e^{\epsilon \|x\|_{L^2}^2}
\bigl(1+\|(-A)^{\frac{1}{4}+\delta}x\|_{L^2}^{16}\bigr)\|(-A)^{-\alpha} g\|_{L^2} \|(-A)^{-\beta} h\|_{L^2}.
\]
In the above, $\delta,\epsilon\in(0,\infty)$ denote arbitrarily small auxiliary parameters, and $C_{\alpha,\delta,\epsilon}(T,Q,\varphi)$ and $C_{\alpha,\beta,\delta,\epsilon}(T,Q,\varphi)$ 
are positive real numbers, which are independent of $t\in(0,T]$ and of $x,g,h$. To the best of our knowledge, 
the regularity properties above on solutions to Kolmogorov equations associated with the stochastic Burgers equation driven by additive trace-class noise are new.

Regularity results for Kolmogorov equations associated with SPDEs have been considered first in the context of numerical analysis for SPDEs in~\cite{Debussche:2011},
for parabolic semilinear SPDEs with globally Lipschitz nonlinearities. Similar results have been obtained in~\cite[Theorem 1.1]{AnderssonEtAl:2019}. 
Both~\cite{Debussche:2011} and~\cite{AnderssonEtAl:2019} deal with multiplicative (space-time) white noise, which requires stronger conditions on $\alpha$ and $\beta$:
one needs $\alpha<1/2$ and $\alpha+\beta<1/2$, respectively. Using a novel expression for the derivatives of the solution to a Kolmogorov equation  and Malliavin calculus techniques,
these conditions on $\alpha$ and $\beta$ were weakened in~\cite{MR3862147}. 
Note that~\cite{MR3862147} deals with SPDEs where the term  `$+\nabla(X(t,z)^2)$' in~\eqref{eq:Burgers_intro} 
is replaced by `$+\nabla (b(X(t,z))$' for a smooth and globally Lipschitz mapping $b$ (and the additive trace-class noise in~\eqref{eq:Burgers_intro} 
is replaced by multiplicative space-time white noise). Despite the fact that~\cite{MR3862147} does not allow quadratic growth in $b$ 
(as is the case in our setting), it turns out to be possible to adapt the strategy from~\cite{MR3862147} to deal with the quadratic growth of the nonlinearity of the Burgers equation, 
the details are given in Section~\ref{sec:Kolmogorov} (with no need of Malliavin calculus techniques).\par 

The strategy introduced in~\cite{MR3862147} has also been employed in~\cite{MR4132896} to prove similar regularity properties for the 
solutions to Kolmogorov equations associated to the stochastic Allen--Cahn equation, and in~\cite{MR3984308} for a more general class of equations with monotone nonlinearities.
Recently, this strategy has also been applied in~\cite{BrehierCuiWang:2022} for the solutions to Kolmogorov equations associated to the stochastic Cahn--Hilliard equations.
Compared with the works cited above, note that we have some exponential dependence with respect to $\|x\|_{L^2}^2$ (with arbitrarily small $\epsilon\in(0,\infty)$), which is a specific feature of the stochastic Burgers equation, 
and which explains the need to deal with exponential moment bounds for the solutions of~\eqref{eq:SPDE} and its Galerkin approximation~\eqref{eq:Galerkin-Burgers}.

Other types of regularity properties for solutions to Kolmogorov equations associated with the stochastic Burgers equation have been obtained in the literature.
In \cite[Section~4]{MR1722786} it is shown that when  $Q=A^{-\beta}$ with $\beta\in(\frac12,1)$, for $M\in\N$, $t\in [0,T]$, $x,h\in H_M$ and a  smooth function $\varphi$, there exists 
$C_{\epsilon,T,\varphi}\in (0,\infty)$ such that
\[ 
Du_M(t,x).(h)
\le C_{\epsilon,T,\varphi} \, t^{-\frac12-\frac{\beta}{2}}\, e^{\epsilon\|x\|_{L^2}^2}\,   \|h\|_{L^2}. 
\]
Note that  in our setting the operators $A$ and $Q$ need not commute. Moreover, the condition $Q=A^{-\beta}$ with $\beta\in(\frac12,1)$ provides a non-degeneracy condition on the noise, that is not considered in this work. Moreover, \cite[Theorem 2.2 (iv)]{RocknerSobol:2006} relates certain weighted norms of $u(t,\cdot)$ 
to such norms of $\varphi$.

We are aware that a weak error analysis for a \emph{full} discretization of the Burgers equation is desirable 
(i.e., one should consider both temporal and spatial discretizations,
as well as finite-dimensional approximations of the noise). However, obtaining an error analysis for a temporal discretization still requires a substantial amount of work. 
It would also be interesting to establish weak error estimates for approximations of the stochastic Burgers equation driven by space-time white noise. 
However, this would require more advanced techniques. We therefore leave those open questions for future investigations.

The paper is organized as follows. In Section \ref{sec:preliminaries} we describe the
model, recall some basic properties of the operators $A$, of the bilinear term $B$, and some classical inequalities. 
Section \ref{sec:bounds} describes $L^p$moments of the Galerkin approximation $X_M$ in  various functional spaces, among which $L^\infty([0,T]\times [0,1])$ and $C^\gamma([0,T];W^{\lambda,2})$ in terms of moments of the initial condition, as well
as exponential moments of the $L^2$-norm uniformly in time. The corresponding proofs are given in the Appendix \ref{app}. 
In section \ref{sec:Kolmogorov} we prove the regularity properties of the Kolmogorov equation, which a crucial tool to obtain the weak convergence rate proven in Section \ref{weak-rate}.

\section{Preliminaries}\label{sec:preliminaries}

\subsection{Notation}\label{ssec:notation}

Let $\N=\{1,2,\ldots\}$ denote the set of (strictly) positive integers and $\N_0=\{0\}\cup \N$. For any $x\in [0,\infty)$ we denote by $\lfloor x \rfloor = \sup\{ n\in \N_0\colon n\leq x\}$ the integer part of $x$. For any $x,y\in \R$ we set $x\wedge y = \min(x,y)$ and $x\vee y = \max(x,y)$.\par 

Given two Banach spaces $(X,\left\| \cdot \right\|_X)$ and $(Y,\left\| \cdot \right\|_Y)$ we let $(\mathcal{L}(X,Y),\left\| \cdot \right\|_{\mathcal{L}(X,Y)})$ denote the Banach space of bounded linear operators from $X$ to $Y$; we set $\mathcal{L}(X)=\mathcal{L}(X,X)$. The dual of a Banach space $X$ is denoted by $X^*$, and the adjoint of an operator $C\in\mathcal{L}(X,Y)$ is denoted by $C^*$. Given a Banach space $X$, a measure space $(S,\Sigma,\mu)$, and $p\in [1,\infty)$ we let $L^p(S,X)$ denote the Bochner space of strongly measurable, $p$-integrable functions from $S$ to $X$; if $X$ is separable the strong and weak measurability are equivalent by Pettis' theorem.\par

Given a Hilbert space $(H,\langle \cdot, \cdot \rangle)$, we let $\mathcal{L}_1(H)$ denote the space of trace-class operators on $H$ and we let $\mathcal{L}_2(H)$ denote the space of Hilbert--Schmidt operators on $H$.

We let $C^k(H,\R)$ denote the Banach space of $k$-times continuously (Fr\'echet) differentiable functions. 
We denote the (real) Lebesgue spaces on $(0,1)$ by $(L^p,\|\cdot\|_{L^p})$, $p\in [1,\infty]$; the inner product on $L^2$ is denoted by $\langle \cdot , \cdot \rangle_{L^2}$.

We denote the Sobolev spaces with smoothness parameter $k\in \N$ and integrability parameter $p\in [1,\infty]$ by $W^{k,p}$, for a definition see e.g.~\cite[Section 1.11.2]{Yagi:2010}. 
We denote the fractional (real) Slobodeckij--Sobolev spaces on $(0,1)$ with smoothness parameter $\alpha\in (0,\infty)$ and integrability parameter $p\in [1,\infty)$ by 
$(W^{\alpha,p},\| \cdot \|_{W^{\alpha,p}})$ (see e.g.~\cite[Section 2]{NezzaEtAl:2012}),
i.e., for $f\in W^{\lfloor \alpha \rfloor, p}$ we set 
\[  \| f\|_{W^{\alpha ,p}} = \| f\|_{W^{\lfloor \alpha \rfloor,p}}  +  \Big(\int_0^1 \int_0^1 \frac {
\vert f^{\lfloor \alpha \rfloor}(y)-f^{\lfloor \alpha \rfloor}(z)\vert^p}{\vert y-z\vert^{1+(\alpha - \lfloor \alpha \rfloor) p}} \,dy\,dz \Big)^{\frac{1}{p}},
\]
and we set $W^{\alpha,p}=\{ f\in W^{\lfloor \alpha \rfloor, p }\colon \| f \|_{W^{\alpha,p}}<\infty\}$. 
 For $p\in [1,\infty)$ and $\alpha > \frac{1}{p}$ we set $W^{\alpha,p}_0=\{f\in W^{\alpha,p}\colon f(0)=f(1)=0\}$ (this makes sense owing to the Sobolev inequality~\eqref{eq:SobolevII} below). 

Given a function $\phi:H\to\R$ which is twice (Fr\'echet) differentiable, for all $x\in H$ we let $D\phi(x)\colon H \ni h \mapsto D\phi(x).h$ and 
$D^2\phi(x)\colon H\times H \ni (g,h)  \mapsto D^2\phi(x).(g,h)$ 
denote the first and second order derivatives of $\phi$ at $x\in H$.

\subsection{Some inequalities}\label{ssec:inequalities}

We recall some useful inequalities. Firstly, on the domain $(0,1)$ the Poincar\'e inequality:
\begin{equation}\label{eq:poincare}
 \|x\|_{L^2} \leq 
\tfrac{1}{\sqrt{2}} \|\nabla x\|_{L^2},\quad \forall x\in W^{1,2}_0,
\end{equation}
which is an immediate consequence of the fundamental theorem of calculus and H\"older's inequality.

We shall also frequently use the Sobolev inequalities, namely, for all $p\in[1,\infty)$ and $q\in(p,\infty)$, one has $W^{\frac{1}{p}-\frac{1}{q},p}\subseteq L^q$ and there exists $C_{p,q}\in(0,\infty)$ such that 
\begin{equation}\label{eq:SobolevI}
\| x \|_{L^q} \leq C_{p,q} \| x \|_{W^{\frac{1}{p}-\frac{1}{q},p}}, \quad \forall x\in W^{\frac{1}{p}-\frac{1}{q},p}, 
\end{equation}
see, e.g.,~\cite[Theorem 6.5]{NezzaEtAl:2012}.
Moreover, for all $\alpha\in (0,1)$ and all $p\in ((1-\alpha)^{-1},\infty)$, one has $W^{\alpha+\frac{1}{p},p}\subseteq C^\alpha$ and there exists $C_{\alpha,p}\in(0,\infty)$ such that
\begin{equation}\label{eq:SobolevII}
\| x \|_{C^{\alpha}} \leq C_{\alpha,p}\| x \|_{W^{\alpha+\frac{1}{p},p}}, \quad \forall x\in W^{\alpha+\frac{1}{p},p}, 
\end{equation}
see, e.g.,~\cite[Theorem 8.2]{NezzaEtAl:2012}.
Since $W^{1,2} \subseteq C^{\alpha}$ for all $\alpha\in (0,\frac{1}{2})$,  
we easily deduce that $W^{1,2}$ is an algebra: there exists $C\in (0,\infty)$ such that 
\begin{equation}\label{eq:H1_algebra}
\| x_1 x_2 \|_{W^{1,2}} \leq C \| x_1 \|_{W^{1,2}} \| x_2 \|_{W^{1,2}}, \quad \forall x_1,x_2\in W^{1,2}.
\end{equation}
Finally, we recall a special case of the Gagliardo-Nirenberg inequalities (see, e.g.,~\cite[Theorem 1.37]{Yagi:2010}): for all $p\in (2,\infty)$, there exists $C_p\in(0,\infty)$ such that 
\begin{equation}\label{eq:GN}
\| x \|_{L^p} \leq  C_p \| x \|_{L^2}^{\frac{1}{2}+\frac{1}{p}} \| x \|_{W^{1,2}}^{\frac{1}{2}-\frac{1}{p}},\quad  \forall x\in W^{1,2}.
\end{equation}

\subsection{Setting}\label{ssec:setting}
Throughout this paper we fix a terminal  time $T\in(0,\infty)$.

We let $A\colon W^{1,2}_0 \cap W^{2,2} \subseteq L^2 \rightarrow L^2$ denote the Dirichlet Laplace operator on $L^2$, i.e.,
\begin{equation}\label{eq:defLaplace}
A x = - \sum_{k\in \N} (\pi k)^2 \langle x, h_k\rangle_{L^2} h_k,\quad 
\forall x\in W^{2,2}\cap W^{1,2}_0,
\end{equation}
where $h_k=\sqrt{2}\sin(k\pi \cdot)$ for all $k\in\N$. The eigenvectors $\bigl(h_k\bigr)_{k\in\N}$ define a complete orthornormal system of $L^2$. Note that $A$ generates an analytic $C_0$-semigroup $(e^{tA})_{t\geq 0}$ on $L^2$. 

We let $B\colon W^{1,2}\times W^{1,2}\rightarrow L^1$ denote the bilinear operator defined by
\begin{equation}
    B[x_1,x_2]=x_1 \nabla x_2 + x_2 \nabla x_1,\quad \forall x_1,x_2\in W^{1,2}, 
\end{equation}
and we set $B(x)=B[x,x]$ for $x\in W^{1,2}$. Note that an integration by parts yields the identity $\langle B(x),x\rangle_{L^2} =0$ for all $x\in W^{1,2}$. 

We fix a positive trace class self-adjoint linear operator $Q\in \mathcal{L}_1(L^2)$ and we let $(e_k)_{k\in \N}$ be a complete orthonormal system of eigenvectors of $Q$ corresponding to the eigenvalues $(q_k)_{k\in \N}$. The bounded linear operators $Q\in\mathcal{L}_1(L^2)$ and $\sqrt{Q}\in \mathcal L_2(L^2) $ are given by
\[
Qx=\sum_{k\in\N}q_k\langle x,e_k\rangle_{L^2}e_k,\quad \sqrt{Q}x=\sum_{k\in\N}\sqrt{q_k}\langle x,e_k\rangle_{L^2}e_k,\quad \forall~x\in L^2.
\]

We fix a filtered probability space $(\Omega,\cF,\PP,(\cF_t)_{t\in [0,T]})$, which is assumed to be large enough to allow for the existence of a $(\cF_t)_{t\in [0,T]}$-Brownian motion $W^Q\colon [0,T]\times \Omega \rightarrow L^2$ with covariance operator $Q$: this means that there exists a sequence of independent standard $(\cF_t)_{t\in [0,T]}$-Brownian motions $(W^{(k)})_{k\in \N}$ such that 
\begin{equation}\label{eq:defWQ}
   W^Q(t) = \sum_{k\in \N} \sqrt{q_k} W^{(k)}(t) e_k ,\quad \forall  t\in [0,T]. 
\end{equation}

For every $p\in [4,\infty)$ and every $\cF_0$-measurable $X_0 \in L^p(\Omega,L^2)$ there exists an unique continuous (up to versions) $(\cF_t)_{t\in [0,T]}$-adapted process $X\colon [0,T]\times \Omega \rightarrow H$ such that $\PP(X_t\in W^{1,2}_0)=1$ for all $t\in [0,T]$,
\begin{equation*}
\E\Bigl[ \sup_{t\in [0,T]}\|X(t)\|_{L^2}^{p}+
\int_{0}^{T}\|\nabla X(t)\|_{L^2}^2 \,dt\,\Bigr]<\infty,
\end{equation*}
and 
\begin{equation}\label{eq:Burgers}
X(t) = X_0 + \int_0^{t} \bigl[ A X(s) + B(X(s)) \bigr]\,ds + W^{Q}(t), \quad \forall t\in [0,T].
\end{equation}

The above follows from~\cite[Theorem 1.1 and Remark 3.1]{LiuRockner:2010} (in the notation of~\cite{LiuRockner:2010} we take $V=W^{1,2}_0$, $H=L^2$, $\rho(x) = \| x \|_{L^4}^4$, $\alpha=\beta=2$, $\theta =1$, $f_t \equiv 2\vee \| Q \|_{L_1(H)}$, $K=2$; by
~\eqref{eq:GN} there exists a $C\in (0,\infty)$ such that $\rho(x) \leq C \| x \|_{L^2}^{2}\| x \|_{W^{1,2}}^2$ for all $x\in W^{1,2}_0$). We refer to the process $X$ as the \emph{solution to the Burgers equation (with initial value $X_0$)}.\par

For $M\in\N$, set
\[
H_M = \operatorname{span}(h_1,\ldots,h_M)\subseteq W^{1,2}_{0} \cap W^{2,2}
\]
and let
 $P_M\in \mathcal{L}(L^2)$ denote the orthogonal projection onto $H_M$.
 We define the linear operator $A_M\in \mathcal{L}(L^2,H_M)$ and the bilinear operator $B_M\colon L^2 \times L^2 \rightarrow H_M $
by $A_M = P_M A P_M $ ($=A P_M$) and 
\begin{equation}\label{eq:defB_M}
B_M[x_1,x_2]= P_M B[P_Mx_1, P_Mx_2], \quad \forall x_1,x_2 \in L^2.
\end{equation}
We set $B_M(x)= B_M[x,x]$ for all $x\in L^2$.
It again follows from~\cite[Remark 3.1 and Theorem 1.1]{LiuRockner:2010} (see also~\eqref{eq:Burgerslocmono} below, or see~\cite[Theorem 3.1.1]{PrevotRockner:2007}) that for every $p\in [4,\infty)$ and every $\cF_0$-measurable $X_0 \in L^p(\Omega,L^2)$ there exists unique (up to versions) stochastic process $X_M\colon [0,T]\times \Omega \rightarrow H_M$ such that
\begin{equation}\label{eq:GalsolBurgers}
    X_M(t) = P_M X_0 + 
    \int_{0}^{t}\bigl[ A_M X_M(s) + B_M(X_M(s))\bigr]\,ds + P_MW^Q(t),\quad \forall  t\in [0,T],
\end{equation} 
(in the notation of~\cite{LiuRockner:2010}, we take $V=(H_M,\| \cdot \|_{W^{1,2}})$, $H=(H_M, \| \cdot \|_{L^2})$, $\rho(x) = \| x \|_{L^4}^4$, $\alpha=\beta=2$, $\theta =1$, $f_t \equiv  2 \vee \| Q \|_{L_1(H)}$, $K=2$; note that \cite[Hypothesis 4]{LiuRockner:2010} can be verified uniformly in $M$ using~\eqref{eq:PMnabla} below).
Moreover, the solution $X_M$ can be written using the following  mild formulation: 
\begin{equation}\label{eq:mildGalsolBurgers}
    X_M(t) = e^{tA} P_M X_0 + 
    \int_{0}^{t} e^{(t-s)A} B_M(X_M(s))\,ds + 
    \int_0^t e^{(t-s)A} P_M\,dW^Q(s),\quad \forall  t\in [0,T].
\end{equation}
We refer to the processes $X_M$, $M\in \N$, as the \emph{Galerkin approximations} of $X$. 

\subsection{Properties of the operators \texorpdfstring{$A$, $B$, and $P_M$}{\emph{A}, \emph{B}, \emph{PM}}}\label{ssec:properties}

For any $\alpha \in [0,\infty)$ we can define the fractional powers $(-A)^{\alpha} \colon D((-A)^{\alpha}) \subseteq L^2 \rightarrow L^2$ of $-A$, by
\begin{equation*}
D((-A)^{\alpha}) = \Big\{ x \in L^2 \colon  \sum_{k\in \N} (\pi k)^{4\alpha} \langle x, h_k\rangle_{L^2}^2 < \infty\Big\} 
\end{equation*}
and
\begin{equation}\label{eq:def_fracpowA}
    (-A)^{\alpha}x = \sum_{k\in \N} (\pi k)^{2\alpha} \langle x, h_k\rangle_{L^2} h_k, \quad \forall  x\in D((-A)^{\alpha}).
\end{equation}
Furthermore, we define $(-A)^{-\alpha} \in \mathcal{L}(((-A)^{\alpha})^*,L^2)$ to be the adjoint of $A^{\alpha}$. Parseval's identity implies that
\begin{equation}\label{eq:A_fracpownorminc}
\| (-A)^{\alpha} x \|_{L^2} \leq  \| (-A)^{\beta} x \|_{L^2}
\end{equation} 
for all $\alpha,\beta\in (-\infty,\infty)$ satisfying $\alpha<\beta$ and all $x\in  D((-A)^{\beta\vee 0})$. 

\begin{propo}\label{prop:DA_Sobolev_equiv} 
The fractional domains of $-A$ satisfy
\begin{equation}\label{eq:equinorm}
D((-A)^{\alpha}) = 
\begin{cases}
    W^{2\alpha,2}; & \alpha \in (0,\tfrac{1}{4}),\\
    W^{2\alpha,2}_{0}; & \alpha \in (\tfrac{1}{4},1)
\end{cases}
\end{equation}
with equivalence in norms.
\end{propo}

\begin{proof}
It follows from~\cite[Theorem 16.12]{Yagi:2010} 
that 
\begin{equation}
D((-A_2)^{\alpha}) = 
\begin{cases}
    H^{2\alpha}; & \alpha \in (0,\tfrac{1}{4}),\\
    H^{2\alpha}_{p,D}; & \alpha \in (\tfrac{1}{4},1).
\end{cases}
\end{equation}  
with equivalence in norms, where $H^{2\alpha}_2$ is the Lebesgue-Sobolev space defined in~\cite[Sections 1.11.3-4]{Yagi:2010}, and 
$H^{2\alpha}_{2,D} = \{ f\in H^{2\alpha}_{2} \colon f(0)=f(1)=0\}$, see~\cite[p.\ 560]{Yagi:2010}.
Next, note that the definition of the spaces $H^{s}_p$ ($s\in [0,\infty)$, $p\in [1,\infty)$) in~\cite[Section 1.11.3-4]{Yagi:2010} coincides with the definition of these spaces in~\cite{Triebel:1978}, see~\cite[Theorem 2.3.3(a) and Definition 4.2.1]{Triebel:1978}. Finally,~\cite[Theorem 4.6.1]{Triebel:1978} ensures $H^{s}_2=B^{s}_{2,2}$ for $s\in (0,\infty)\setminus \N$, and from~\cite[4.4.1/Remark 2]{Triebel:1978} we obtain that $B_{p,p}^s = W^{s,p}$ for $s\in (0,\infty)\setminus \N$ (note also that $H^{k}_2=W^{k,2}$ for $k\in \N$, see~\cite[p.41]{Yagi:2010}).
\end{proof}

In view of Proposition~\ref{prop:DA_Sobolev_equiv} above, owing to the Sobolev inequalities~\eqref{eq:SobolevII}, and to the inequality~\eqref{eq:A_fracpownorminc}, for all $\delta\in(0,\infty)$, one has $D((-A)^{\frac{1}{4}+\delta})\subseteq L^{\infty}$ and
there exists $C_\delta\in(0,\infty)$ such that
\begin{equation}\label{eq:Linfty_DA_bound}
    \|x\|_{L^\infty}\le C_\delta \|(-A)^{\frac14+\delta}x\|_{L^2}, \quad \forall~x\in D((-A)^{\frac14+\delta}).
\end{equation}
By a duality argument, one obtains the inequality
\begin{equation}\label{eq:Linfty_DA_bound-dual}
    \|(-A)^{-(\frac14+\delta)}y\|_{L^2}\le C_{\delta}\|y\|_{L^1}, \quad \forall y\in L^1. 
\end{equation}

The following result summarizes some well-known properties of the fractional powers of $A$, and of the semigroup $\bigl(e^{tA}\bigr)_{t\ge 0}$, see e.g.~\cite[Chapter 2.6]{Pazy:1983}. In our setting the proof is elementary.

\begin{lemma}\label{lem:A_analytic}
For all $\alpha \in (0,\infty) $, $\beta \in [0,1]$, $t\in (0,\infty)$, and $x\in L^2$ one has
\begin{enumerate}
    \item $\| (-A)^{\alpha}  e^{t A} x \|_{L^2} \leq e^{\alpha(\log(\alpha)-1)} t^{-\alpha} \| x \|_{L^2}$,
    \item $\| (-A)^{-\beta}(e^{tA}- I) x \|_{L^2} \leq t^{\beta} \| x \|_{L^2}$.
\end{enumerate}
\end{lemma}

\begin{proof}
Those inequalities are immediate consequences of the fact that $A$ is diagonizable with $\sigma(A)\subseteq (-\infty,0)$ ( one has $Ah_k=-(\pi k)^2 h_k$ for all $k\in\N$, see~\eqref{eq:defLaplace}),  and of the elementary inequalities 
$\underset{\lambda \in (0,\infty)}\sup~\lambda^{\alpha} e^{-\lambda t} \leq e^{\alpha(\log(\alpha)-1)} t^{-\alpha}$ and 
$\underset{\lambda \in (0,\infty)}\sup~\lambda^{-\beta} (1-e^{-\lambda t}) \leq t^{\beta} $.
\end{proof}

Let us now provide a useful product inequality.

\begin{lemma} \label{lem:fracpow_B_est}
For all $\gamma\in(0,\frac18)$, there exists $C_\gamma\in(0,\infty)$ such that for all $x_1\in D((-A)^{\frac14+4\gamma})$ and all $x_2\in D((-A)^{4\gamma})$ one has
\[
\|(-A)^\gamma (x_1x_2)\|_{L^2}\le C_{\gamma} \|(-A)^{1/4+4\gamma}x_1\|_{L^{2}} \|(-A)^{4\gamma}x_2\|_{L^{2}}.
\]
\end{lemma}

Note that below Lemma~\ref{lem:fracpow_B_est} is used for arbitrarily small $\gamma$.

\begin{proof}
    Let $p,q\in(1,\infty)$ be such that  $1/p+1/q=1$. 
   Let $x_1\in  D((-A)^{\frac14+4\gamma})$ and all $x_2\in D((-A)^{4\gamma})$.  Owing to~\cite[Inequality~(12)]{MR3862147}, one obtains 
    \[
    \|(-A)^\gamma (x_1x_2)\|_{L^2} \le C_{\gamma,p,q} \|(-A)^{2\gamma}x_1\|_{L^{2p}} \|(-A)^{2\gamma}x_2\|_{L^{2q}}.
    \]
  for some $C_{\gamma,p,q}\in (0,\infty)$ which does not depend on $x_1, x_2$. Let us choose $q=\frac{1}{1-8\gamma}\in(1,\infty)$ and observe that $\frac14-\frac{1}{4q}=2\gamma$. Therefore applying the inequality~\eqref{eq:Linfty_DA_bound}, the Sobolev inequality~\eqref{eq:SobolevI}  and the equivalence of norms property~\eqref{eq:equinorm} one obtains the upper bounds
    \begin{align*}
    \|(-A)^{2\gamma}x_1\|_{L^{2p}}&\le \|(-A)^{2\gamma}x_1\|_{L^{\infty}}\le C_{\gamma} \|(-A)^{\frac{1}{4}+4\gamma}x_1\|_{L^{2}}, \\
    \|(-A)^{2\gamma}x_2\|_{L^{2q}}&\le C_q\|(-A)^{\frac{1}{4}-\frac{1}{4q}+2\gamma}x_2\|_{L^{2}} = C_q\|(-A)^{4\gamma}x_2\|_{L^{2}}.
    \end{align*}
    Gathering the upper bounds concludes the proof. 
\end{proof}

Next let us state some properties related to the behavior of the gradient operator $\nabla$ and of the fractional powers of $-A$, obtained by interpolation.

\begin{lemma}\label{lem:deriv_bdd_fracdomains}
For all $\alpha\in[0,\frac12]$ and all $x\in W^{1,2}$ one has
\begin{equation}
\| (-A)^{-\alpha} \nabla (-A)^{-\frac{1}{2}+\alpha} x \|_{L^2} \leq \| x \|_{L^2}.
\end{equation}
\end{lemma}

\begin{proof}
For all $x\in D((-A)^{\frac12})$, one has $\|\nabla x\|_{L^2}=\|(-A)^{\frac12}x\|_{L^2}$. Therefore one has the property $\nabla\in \mathcal{L}\bigl(D((-A)^{\frac12}),L^2\bigr)$. In addition, by duality one has $\nabla\in \mathcal{L}\bigl(L^2,(D((-A)^{\frac12})^*)\bigr)$. The result then follows from complex interpolation theory and the fact that $A$ has bounded imaginary powers (see, e.g., \cite[Theorems~2.6 and~4.17]{Lunardi:2018}; it is also used that $[L^2,(D((-A)^{\frac12})^*]_{\alpha} = ([L^2,D((-A)^{\frac12})]_{\alpha})^*$, see~\cite{Calderon:1964}). 
\end{proof}

We will also need a similar result for the realization of the operator $A$ in $L^p$: let $A_p\colon D(A_p)\subseteq L^p \rightarrow L^p$ denote the realization of $A$ in $L^p$; by~\cite[Theorem 2.19]{Yagi:2010} this is again the generator of an analytic $C_0$-semigroup. In particular, following~\cite[Section 2.7.7]{Yagi:2010} we can define fractional powers $(-A_p)^{\alpha} \colon D((-A_p)^{\alpha}) \rightarrow L^p$  for all $\alpha \in (0,\infty)$ (see also~\cite[Chapter 2]{Pazy:1983}). We shall need the following:

\begin{lemma}\label{lem:deriv_bdd_sqrtA}
Let $p\in (1,\infty)$. There exists $C_p\in (0,\infty)$ such that
\begin{equation}\label{eq:deriv_bdd_sqrtA}
    \| (-A_p)^{-\frac{1}{2}}\nabla x \|_{L^p} \leq  C_p \| x \|_{L^p}, \quad \forall x \in W^{1,p}.
\end{equation}
\end{lemma}
\begin{proof}
Note that $A_p$ has a bounded $H^{\infty}$-calculus by e.g.~\cite[Theorem 1.1]{LindemulderVeraar:2020} (take $\gamma=0$) or~\cite[Section 8.2 and Theorem 10.15]{KunstmannWeis:2004}.
It thus follows from~\cite[Theorem 16.15 and Section 1.11.2]{Yagi:2010} 
that $D((-A_p)^{\frac{1}{2}}) = W^{1,p}_0 $, i.e., $(-A_p)^{-\frac{1}{2}}\in L(L^2,W^{1,p}_0)$. Let $q\in (1,\infty)$ be such that $\frac{1}{p}+\frac{1}{q}=1$. By definition, one has $\nabla \in \mathcal{L}(W_0^{1,q},L^{q})$, i.e., $\nabla^* \in \mathcal{L}(L^{p},(W_0^{1,q})^*)$. As $\nabla^* x = - \nabla x$ this completes the proof.
\end{proof}

As for the operator $B$, recall that for $x_1,x_2\in W^{1,2}_0 $ one has 
\begin{align}\label{eq:Bsym}
    \langle x_1, B[x_1,x_2] \rangle_{L^2} 
    &
    = \int_{0}^{1} x_1(\theta) \nabla (x_1x_2)(\theta) \,d\theta
    = [ x_1^2 x_2 ]^{\theta=1}_{\theta=0} - 
    \int_{0}^{1}( \nabla x_1)(\theta) (x_1x_2)(\theta)\,d\theta  \nonumber 
   \\ & = -\tfrac{1}{2} \langle x_2, B(x_1) \rangle_{L^2}.
\end{align}
Note that for $B_M$ as defined in~\eqref{eq:defB_M}, using~\eqref{eq:Bsym} one has
\begin{align}\label{eq:BsymM}
    \langle x_1, B_M[x_1,x_2] \rangle_{L^2} 
    &    
    =\langle x_1, P_M B[P_M x_1, P_M x_2] \rangle_{L^2}
    = -\tfrac{1}{2} \langle x_2, B_M(x_1) \rangle_{L^2}.
\end{align}
Thus, we deduce that
\begin{equation}\label{eq:BsymM2}
    \langle x, B_M(x) \rangle_{L^2}=0,\quad \forall~x\in H_M.
\end{equation}
Finally, on several occasions we shall need the following observation:
\begin{align}
P_M \nabla x &= \nabla Q_M x, \quad x\in W^{1,2},\, M\in \N, \label{eq:PMnabla}\\
\nabla P_M x & = Q_M \nabla x, \quad x\in W^{1,2}_0,\, M\in \N,\label{eq:nablaPM}
\end{align}
where $Q_M\in \mathcal{L}(L^2)$ is the orthogonal projection onto $\operatorname{span}\{ \cos(k\pi \cdot) \colon k\in \{1,2,\ldots,M\}\}$.

\section{(Exponential) moment bounds and regularity properties}\label{sec:bounds}

Moment bounds for (a Galerkin approximation of) the stochastic Burgers equation have been established in e.g.~\cite[Proposition 2.1]{MR1722786} and~\cite[Lemma 2.2]{LiuRockner:2010}. Moreover, exponential moment bounds for (the Galerkin approximation of) the stochastic Burgers equation have been established in e.g.~\cite[Equation (5.5)]
{CoxHutzenthalerJentzen:2024},~\cite[Proposition 2.2]{MR1722786}. However, to the best of our knowledge the precise bounds that we needed to obtain the necessary regularity results of the associated Kolmogorov 
equation are \emph{not} yet available in the literature (for example,~\cite[Lemma 2.2]{LiuRockner:2010} deals with a different projection of the noise,~\cite{CoxHutzenthalerJentzen:2024} and~\cite{MR1722786} take deterministic initial values, etc.). We thus provide the required bounds in Lemmas~\ref{lem:expest} and~\ref{lem:X_sup_moments} below.

As the proofs of Lemmas~\ref{lem:expest},~\ref{lem:momLinfty} and~\ref{lem:X_sup_moments} are technical but follow from classical arguments, they are postponed to Appendix~\ref{app}.

\begin{lemma}\label{lem:expest}
\begin{enumerate}[(i)]
\item\label{it:momentbounds} For all $p\in [4,\infty)$ there exists $C_{p}(T,Q)\in (0,\infty)$ such that if the initial condition satisfies $\E [ \| X_0 \|_{L^2}^{p}]<\infty$, then
\begin{equation}\label{pmoments}
\sup_{M\in \N} 
\E \Bigl[ \sup_{t\in [0,T]}\|X_M(t)\|_{L^2}^{p}+
\int_{0}^{T}\| X_M(t) \|_{L^2}^{p-2} \|\nabla X_M(t)\|_{L^2}^2 \,dt\,\Bigr] \leq C_p(T,Q) \E [ \| X_0 \|_{L^2}^{p}].
\end{equation}
\item\label{it:expmomentbounds} Assume that there exists a $\gamma_0\in (0,\infty)$ such that the initial value $X_0$ satisfies the exponential moment bound $ \E [\exp(\gamma_0 \| X_0\|_{L^2}^2 ) ]< \infty$.
For all $\beta \in (0, \frac{\gamma_0}{1 + 2 \gamma_0  \| Q \|_{\mathcal{L}(L^2)}} )$  
 one has
\begin{equation}  \label{expmoments_ranIV}
\begin{aligned}
& \underset{M\in \N}\sup~  \E\Bigl[ \exp\Bigl(\beta \underset{0\leq t\leq T}\sup~\|X_M(t)\|_{L^2}^2
   + \beta \int_0^T\!\!  \| \nabla X_M(s)\|_{L^2}^2 \,ds \Bigr) \Bigr]\\
& \qquad \leq    2 e^{\beta T \Tr (Q)} \bigl( \E\bigl[\exp\bigl(\gamma_0 \|X_0\|_{L^2}^2 \bigr)\bigr]\bigr)^{\frac{\beta}{\gamma_0}}.
\end{aligned}
\end{equation}

\end{enumerate}
 \end{lemma}

Note that if the initial value $X_0$ is deterministic, i.e., if there exists $x_0\in L^2$ such that $\PP(X_0=x_0)=1$, then for all $\beta\in (0,\frac{1}{2 \| Q \|_{\mathcal{L}(L^2)}})$  one has
\begin{equation}  \label{expmoments_detIV}
 \underset{M\in \N}\sup~  \E\Bigl[ \exp\Bigl(\beta \underset{0\leq t\leq T}\sup~\|X_M(t)\|_{L^2}^2
   + \beta \int_0^T\!\!  \| \nabla X_M(s)\|_{L^2}^2 \,ds \Bigr) \Bigr]
\leq    2 e^{\beta T \Tr (Q)} \exp\bigl(\beta \|x_0\|_{L^2}^2 \bigr).
\end{equation}

We next provide some moment bounds for the $L^\infty$ norm of $X_M$, this is an intermediate step to obtain Lemma~\ref{lem:X_sup_moments}.
\begin{lemma}\label{lem:momLinfty}
    For all $\alpha \in ( \frac{1}{4}, \frac{1}{2})$ and $p\in   [\frac{8}{3},\infty)$,
    there exists $C_{p,\alpha}(T,Q)\in(0,\infty)$, such that if the initial value satisfies the conditions
     $X_0\in D((-A)^{\alpha})$ a.s.\ and $\E\bigl[\| (-A)^\alpha X_0 \|_{L^2}^{p}\bigr] +  \E\bigl[\| X_0 \|_{L^2}^{3p}\bigr]<\infty$,
    then 
\begin{align}\label{mom_X_infty}
\underset{M\in \N}\sup~ \E\bigl[ \underset{t\in [0,T], z \in [0,1]}\sup~ | X_M(t)(z) |^{p} \bigr]
& \leq 
C_{p,\alpha}(T,Q)\Bigl(1+\E\bigl[ \| (-A)^\alpha X_0 \|_{L^2}^{p} \bigr]+  \E\bigl[ \| X_0 \|_{L^2}^{3p}\bigr] \Bigr).
\end{align}
\end{lemma}

Finally, the following lemma provides moment bounds for the space-time regularity of $X_M$.

\begin{lemma}\label{lem:X_sup_moments}
Let $\alpha \in ( \frac{1}{4}, \frac{1}{2})$ and $p\in  [\frac{4}{3},\infty)$.
For all $\lambda,\gamma \in  (0,\frac{1}{2})$ satisfying the condition $\lambda + \gamma < \alpha $, there exists $C_{\alpha,\gamma,\lambda,p}(T,Q) 
\in (0,\infty)$ such that if the initial value satisfies the conditions $X_0 \in D((-A)^{\alpha})$ a.s.\ and $\E\bigl[\| (-A)^\alpha X_0 \|_{L^2}^{2p}\bigr] +  \E\bigl[\| X_0 \|_{L^2}^{6p}\bigr]<\infty$, then one has
\begin{align} \label{eq:XM_moment_mixed_reg}  
\underset{M\in \N}\sup~ \E \bigl[ \| (-A)^{\lambda} X_M \|^p_{C^{\gamma}([0,T], L^2)}\bigr]
&\leq C_{\alpha,\gamma,\lambda,p}(T,Q)
\Bigl( 1+\E \bigl[\| (-A)^\alpha X_0 \|_{L^2}^{2p}\bigr]+  \E \bigl[\| X_0 \|_{L^2}^{6p} \bigr]\Bigr).
\end{align}
\end{lemma}

\begin{rem}\label{rem:notoptimal}
The integrability conditions imposed on the initial value $X_0$ in Lemmas~\ref{lem:momLinfty} and~\ref{lem:X_sup_moments} may not be optimal, e.g.,\ in order to obtain moments of order $p$ for $X_M$, assuming that $\|X_0\|_L^2\in L^{3p}(\Omega,L^2)$ or $\|X_0\|_L^2\in L^{6p}(\Omega,L^2)$ for~\eqref{mom_X_infty} and~\eqref{eq:XM_moment_mixed_reg} respectively is sufficient in the proofs given in Appendix~\ref{app} but may not be necessary.
\end{rem}

\section{Regularity properties for solutions of the associated Kolmogorov equations}\label{sec:Kolmogorov}

In order to obtain estimates for the weak error, we need to analyse the regularity properties of solutions to the Kolmogorov equation associated with the Galerkin approximation $X_M$ of the stochastic Burgers equation, and to obtain bounds which are uniform with respect to $M\in\N$.

For any $M\in\N$ and any $x\in H_M$, the solution of~\eqref{eq:GalsolBurgers} with initial value $X_M(0)=x$ is denoted by $\bigl(X_M^x(t)\bigr)_{t\ge 0}$. Given a $T\in(0,\infty)$ and a function $\varphi \in C^2(L^2,\R)$ with bounded first and second order derivatives, we define the function $u_M\colon [0,T]\times H_M \rightarrow \R$ by
\begin{equation}\label{eq:Kolmogorov}
u_M(t,x) = \E[\varphi(X_M^x(t))],  \quad \forall~t\in [0,T],\, x\in H_M. 
\end{equation} 
By classical results on the Kolmogorov backward equation (see, e.g.,~\cite[Theorem 4.8.6]{KloedenPlaten:1992} or~\cite[Theorem 7.5.1]{DaPratoZabczyk:2002} combined with a localization argument
(to deal with the fact that both $B_M$ and $DB_M$ are not bounded) one has $u_M \in C^{1,2}([0,T]\times H_M, \R)$. Moreover, the mapping $u_M$ is solution to the partial differential equation
\begin{equation}\label{eq:Kolmogorov_eqn}
\left\lbrace
\begin{aligned}
\tfrac{\partial u_M}{\partial t}(t,x)
& =
\tfrac{1}{2}\sum_{j\in\N} q_j D^2 u_M(t,x).(P_M e_j, P_M e_j)
+ Du_M(t,x).(Ax+B_M(x)),\quad   t\in [0,T],\, x\in H_M;
\\
u(0,x) &= \varphi(x), \quad x\in H_M,
\end{aligned}\right.
\end{equation}
which is referred to as the Kolmogorov equation associated with the finite dimensional stochastic evolution equation~\eqref{eq:GalsolBurgers}. Above the first and second derivatives $Du_M(t,x)$ and $D^2u_M(t,x)$ are considered with respect to the variable $x\in H_M$, see the notation introduced in Section~\ref{ssec:notation}.

The main result of this section is the following result on the first and second order derivatives $Du_M(t,x)$ and $D^2u_M(t,x)$.

\begin{theo}\label{theo:Kol_reg}
Let $T\in(0,\infty)$ and assume that $\varphi\in C^2(L^2,\R)$ has bounded first and second order derivatives.

(i) Let $\alpha \in [0,1)$, and let $\delta,\epsilon \in (0,\infty)$ be arbitrarily small auxiliary parameters.
There exists  $C_{\alpha,\delta,\epsilon}(T,Q,\varphi) \in (0,\infty)$ such that for all 
$M\in \N$, all $x,h \in H_M$, and all $t\in (0,T]$ one has 
\begin{align}\label{eq:Kol_deriv_est}
|Du_M(t,x).(h)|
& \le C_{\alpha,\delta,\epsilon}(T,Q,\varphi) \,  t^{-\alpha} \,  e^{\epsilon \|x\|_{L^2}^2} \,  \big( 1+\|(-A)^{\frac{1}{4}+\delta } x\|_{L^2}^{6}\big)\, 
 \|(-A)^{-\alpha}h\|_{L^2}.
\end{align}

(ii) Let $\alpha,\beta\in [0,1)$ be such that $\alpha+\beta<1$, and let $\delta, \epsilon\in (0,\infty)$ be arbitrarily small auxiliary parameters. There exists  $C_{\alpha,\beta,\delta,\epsilon}(T,Q,\varphi)\in (0,\infty)$ such that for all $M\in \N$, all $x,g,h \in H_M$, and all $t\in (0,T]$ one has
\begin{align}\label{eq:Kol_2nd_deriv_est}
& |D^2 u_M(t,x).(g,h)| \notag \\ 
& \quad 
\le C_{\alpha,\beta,\delta,\epsilon}(T,Q,\varphi)\, t^{-(\alpha + \beta)}\,   e^{\epsilon \|x\|_{L^2}^2}\,   \big( 1+\|(-A)^{\frac{1}{4}+\delta }x\|_{L^2}^{16}\big) \;
\, \|(-A)^{-\alpha}g\|_{L^2}\|(-A)^{-\beta}h\|_{L^2}.
\end{align}
\end{theo}

Let us describe the strategy for the proof of Theorem~\ref{theo:Kol_reg}. It is based on several intermediate results that we shall present below. The first order derivative $D u_M(t,x).(h)$ can be expressed as follows: for all $M\in \N$, $x,h\in H_M$, and $t\ge 0$, one has
\begin{align}\label{eq:Kol_deriv}
D u_M(t,x) .(h) =\E\bigl[D\varphi(X_M^{x}(t)) . (\eta_M^h(t))\bigr],
\end{align}
where the process $t\in[0,\infty)\mapsto \eta_M^h(t)\in H_M$ is the solution to the linear evolution equation
\begin{equation}\label{eq:X_deriv_initialvalue}
\left\lbrace
\begin{aligned}
&\frac{d}{dt} \eta_M^h(t)=A\eta_M^h(t)+2B_M\bigl[X_M^{x}(t),\eta_M^h(t)\bigr],\quad t\geq 0,\\
&\eta_M^h(0)=h.
\end{aligned}
\right.
\end{equation}
This is proven in e.g.~\cite[Theorem 7.3.6]{DaPratoZabczyk:2002} for the case that the drift and diffusion coefficients are bounded and have bounded derivatives; we can reduce to this setting by localization.
To simplify notation, the dependence of $\eta_M^h(t)$ with respect to $x$ is not mentioned, and the same applies to the other auxiliary processes introduced below. More generally, for any $s\ge 0$, we introduce the process $t\in[s,\infty) \mapsto \eta_M^h(t|s)$ defined on $[s,\infty)$ by
\begin{equation}\label{eta(t|s)}
\left\lbrace
\begin{aligned}
&\frac{d}{dt} \eta_M^h(t|s) = A \eta_M^h(t|s) + 2 B_M\bigl[X_M^x(t), \eta_M^h(t|s)\bigr], \quad t\ge s,\\
&\quad \eta_M^h(s|s)=h.
\end{aligned}
\right.
\end{equation}
Finally, we define the random linear operators $\Pi_M(t,s)\in\mathcal{L}(H_M)$ for all $t\ge s$ given by
\[
\Pi_M(t,s)h=\eta_M^h(t|s),\quad \forall~t\ge s\ge 0,~\forall x,h\in H_M.
\]
Observe that one has the identity $\eta_M^h(t)=\Pi_M(t,0)h$. The solution to~\eqref{eta(t|s)} has the mild formulation
\begin{equation}\label{eq:eta_mild}
\eta_M^h(t|s)=e^{(t-s)A}h+2\int_{s}^{t}e^{(t-r)A}B_M\bigl[X_M^x(r), \eta_M^h(r|s)\bigr]\,dr,\quad \forall~t\ge s.
\end{equation}
In view of this mild formulation and the smoothing property of $e^{tA}$ (see Lemma~\ref{lem:A_analytic}), the term $t^{-\alpha}\|(-A)^{-\alpha}h\|_{L^2}$ on the right-hand side of~\eqref{eq:Kol_deriv_est}
is to be expected. However, the mild formulation above is not sufficient to prove~\eqref{eq:Kol_deriv_est} and obtain the required properties for the linear operators $\Pi_M(t,s)$. We thus define
\begin{equation}\label{tilde-eta}
\tilde{\eta}_M^h(t|s) = \eta_M^h(t|s) - e^{(t-s) A}h, \quad t\geq s. 
\end{equation}
Properties of the linear operators $\Pi_M(t,s)$ are then obtained by writing
\begin{equation}\label{eq:pi_etas}
\Pi_M(t,s)h=\eta_M^h(t|s)=e^{(t-s) A}h+\tilde{\eta}_M^h(t|s),
\end{equation}
and it is thus necessary to understand how $\tilde{\eta}_M^h(t|s)$ depends on $h$. It is straightforward to check that $t\in[s,\infty)\mapsto \tilde{\eta}_M^h(t|s)\in H_M$ is solution to the evolution equation
\begin{equation}\label{tildeeta(t|s)}
\left\lbrace
\begin{aligned}
&\frac{d}{dt} \tilde{\eta}_M^h(t|s) = A \tilde{\eta}_M^h(t|s) + 2 B_M\bigl[X_M^x(t), \tilde{\eta}_M^h(t|s)\bigr]+ 2 B_M\bigl[X_M^x(t), e^{(t-s) A}h\bigr], \quad t\ge s,\\
&\quad \tilde{\eta}_M^h(s|s)=0.
\end{aligned}
\right.
\end{equation}
The variation of constants formula implies the following expression for $\tilde{\eta}_M^h(t|s)$: 
\begin{equation}		\label{eq:VoC_tildeeta}
\tilde{\eta}_M^h(t|s) = 2\int_s^t \Pi_M(t,r) B_M\bigl[X_M^x(r) , e^{(r-s)A}h\bigr] \,dr, \quad  \forall~t\geq s.
\end{equation}
Let us now explain the strategy for the analysis of the second order derivative $D^2 u_M(t,x). (g,h)$, which can be expressed as follows: for all $M\in \N$, $x,g,h\in H_M$, and $t\ge 0$, one has
\begin{align}\label{eq:Kol_2nd_deriv}
D^2 u_M(t,x) . (g,h) =\E\bigl[D^2\varphi(X_M^{x}(t)) . (\eta_M^g(t),\eta_M^h(t))\bigr]
+ \E\bigl[ D\varphi(X_M^{x}(t)) . (\zeta_M^{g,h}(t))\bigr],
\end{align}
where the processes $\bigl(\eta_M^g(t)\bigr)_{t\ge 0}$ and $\bigl(\eta_M^h(t)\bigr)_{t\ge 0}$ are solutions of~\eqref{eq:X_deriv_initialvalue} with initial values $g$ and $h$ respectively, and the process $t\ge 0\mapsto \zeta_M^{g,h}(t)\in H_M$ is solution to the evolution equation
\begin{equation}\label{eq:X_2nd_deriv_initialvalue}
\left\lbrace
\begin{aligned}
&\frac{d}{dt}\zeta_M^{g,h}(t)=A\zeta_M^{g,h}(t)+ 2 B_M\bigl[X_M^x(t),\zeta_M^{g,h}(t)\bigr]
+2B_M\bigl[\eta_M^g(t),\eta_M^h(t)\bigr]
,\;  t\geq 0,\\
&\zeta_M^{g,h}(0)=0.
\end{aligned}
\right.
\end{equation}
The linear operators $\Pi_M(t,s)$ play also a fundamental role for the analysis of the second order derivative. Indeed, applying again the variation of constants formula yields
 the following expression of $\zeta_M^{g,h}(t)$: 
\begin{equation}\label{eq:VOC_zeta}
 \zeta_M^{g,h}(t) = 2 \int_0^t \Pi_M(t,s) B_M\bigl[ \eta_M^g(s), \eta_M^h(s)\bigr] \,ds, \quad \forall~t\geq s.
\end{equation}

In view of the expressions~\eqref{eq:VoC_tildeeta} and~\eqref{eq:VOC_zeta} for the first and second order derivatives, we conclude that it is necessary to derive bounds on the operators $\Pi_M(t,s)$ for $0\leq s < t \leq T$. Indeed, we state and prove three auxiliary results below; these upper bounds reveal the importance of the (exponential) moment bounds and regularity properties for the solution $X_M^x(t)$ obtained in Section~\ref{sec:bounds}.

First, Lemma~\ref{lem_firstder_1} gives an upper bound for $\|\Pi_M(t,s)h\|_{L^2}$ depending on $\|(-A)^{-\alpha}h\|_{L^2}$ in the range $\alpha\in[0,\frac12)$. 
The techniques of the proof are based on classical energy inequalities. To go further and obtain estimates in the range $\alpha\in[\frac12,1)$, see Lemma~\ref{lem_firstder_2}, 
it is necessary to use the expression~\eqref{eq:VoC_tildeeta} for $\tilde{\eta}_M^h(t|s)$. Finally, Lemma~\ref{lem_firstder_3} gives upper bounds for 
$\|(-A)^{\gamma}\Pi_M(t,s)h\|_{L^2}$ with $\gamma\in[0,\frac12)$
depending on $\|(-A)^{-\alpha}h\|_{L^2}$ with $\alpha\in[0,1)$, using the mild formulation of $\eta^h_M(t|s)$. Note that all the upper bounds are given in the almost sure sense. The three auxiliary results are then combined with the (exponential) moment bounds and regularity results from Section~\ref{sec:bounds} to obtain Corollary~\ref{cor_moments_Du} below, and finally this corollary is used to prove Theorem~\ref{theo:Kol_reg}.

\begin{lemma} \label{lem_firstder_1}
Let $\alpha \in [0, \frac{1}{2})$ and  $\epsilon \in (0,\infty)$. There exists  
$C_{\alpha,\epsilon}(T) \in (0,\infty)$ such that for all $M\in \N$, all $x\in H_M$, all $h\in H_M\setminus\{0\}$ and all $0\leq s< t\leq T$, one has
 \begin{equation}	\label{upperPiM-1}
\frac{\| \Pi_M(t,s) h\|_{L^2}}{\| (-A)^{-\alpha} h \|_{L^2}} 
\leq 
 \frac{C_{\alpha,\epsilon}(T)}{(t-s)^{\alpha}}\; e^{ \epsilon\! \int_s^t \!\|\nabla  X_M^x(r)\|_{L^2}^2 \,dr} 
\Big(  \underset{r\in [s,t]}\sup~\|X_M^x(r)\|_{L^\infty}+1\Big). 
\end{equation}
\end{lemma} 
\begin{proof}
Owing to the identity~\eqref{eq:pi_etas} and to the smoothing property stated in Lemma~\ref{lem:A_analytic}, it is sufficient to deal with $\tilde{\eta}_M^h(t|s)$ defined by~\eqref{tilde-eta}.

\noindent Note that $\tilde{\eta}^h_M(t|s) \in H_M$.
 Using standard energy equality techniques applied to the evolution equation \eqref{tildeeta(t|s)} and integrating in time, one obtains
\begin{align*}
&\|\tilde{\eta}_M^h(t|s)\|_{L^2}^2 +2\int_s^t  \langle (-A)\tilde{\eta}_M^h(r|s),\tilde{\eta}_M^h(r|s)\rangle \,dr 
\\
&\qquad =  4\int_s^t 
\Bigl( 
    \langle \tilde{\eta}^h_M(r|s), B[ X_M^x(r), \tilde{\eta}^h_M(r|s)] \rangle_{L^2} 
    + 
    \langle \tilde{\eta}^h_M(r|s), B[ X_M^x(r), e^{(r-s) A} h ] \rangle_{L^2} 
\Bigr) 
\,dr.
\end{align*}
By integration by parts, one has $\langle (-A)\tilde{\eta}_M^h(r|s),\tilde{\eta}_M^h(r|s)\rangle=\|\nabla \tilde{\eta}_M^h(r|s)\|_{L^2}^2$. Note also that owing to~\eqref{eq:Bsym} one has  $\langle x_1, B[x_2,x_1]\rangle_{L^2} = -\frac{1}{2} \langle x_2, B(x_1)\rangle_{L^2} = \frac{1}{2} \langle \nabla x_2 , x_1^2\rangle_{L^2}$ for $x_1,x_2\in H_M\subseteq W^{1,2}_0$.

On the one hand, applying the identity above, the H\"older inequality, and the Gagliardo-Nirenberg inequality~\eqref{eq:GN} (with $p=4$) and the Poincar\'e inequality~\eqref{eq:poincare} (which is applicable as $\tilde{\eta}_M^h(r|s)$ takes values in $H_M\subseteq W^{1,2}_{0}$), one obtains the following upper bounds: there exists $C\in (0,\infty)$ (independent of $x$, $h$, $T$, $r$, $s$, and $M$) such that 
\begin{align*}
\big| \langle \tilde{\eta}^h_M(r|s), B[ X_M^x(r), \tilde{\eta}^h_M(r|s)] \rangle_{L^2} \big| 
&
\leq \tfrac{1}{2}\big| \langle  \nabla X_M^x(r), \tilde{\eta}^h_M(r|s)^2 \rangle_{L^2} \big|
\leq \tfrac{1}{2} \| \nabla X_M^x(r)\|_{L^2} \| \tilde{\eta}^h_M(r|s)\|_{L^4}^2 \\
& 
\leq   C \| \nabla X_M^x(r)\|_{L^2} \| \nabla \tilde{\eta}_M^h(r|s) \|_{L^2}^{\frac12} \|\tilde{\eta}_M^h(r|s)\|_{L^2}^{\frac32}.
\end{align*}
Given an auxiliary parameter $\epsilon\in(0,\infty)$, applying twice the Young inequality, one then obtains the upper bound
\begin{align*}
\big| \langle \tilde{\eta}^h_M(r|s), B[ X_M^x(r), \tilde{\eta}^h_M(r|s)] \rangle_{L^2} \big|
&\leq \tfrac{1}{4}\|  \nabla \tilde{\eta}_M^h(r|s) \|_{L^2}^2 + \tfrac{3}{4}\bigl(C\| \nabla X_M^x(r)\|_{L^2} \|\tilde{\eta}_M^h(r|s)\|_{L^2}^{\frac32}\bigr)^{\frac43}\\
&\leq 
\tfrac{1}{4}\|  \nabla \tilde{\eta}_M^h(r|s) \|_{L^2}^2 
+ \Big( \tfrac{\epsilon}{4} \|\nabla X_M^x(r)\|_{L^2}^2 +  \tfrac{C^4}{\epsilon^2} \Big)  \|\tilde{\eta}_M^h(r|s)\|_{L^2}^2.
\end{align*}\par 
On the other hand, integrating by parts and applying the Cauchy-Schwarz and Young inequalities yields
\begin{align*}
\big| \langle \tilde{\eta}^h_M(r|s), B[ X_M^x(r), e^{(r-s)A} h] \rangle_{L^2} \big| = & \big| \langle \nabla \tilde{\eta}_M^h(r|s) , X_M^x(r) e^{(r-s)A}h \rangle_{L^2} \big| \\
\leq & \tfrac{1}{4} \|  \nabla \tilde{\eta}_M^h(r|s)\|_{L^2}^{2}  + \|X_M^x(r)\|_{L^\infty}^2 \|e^{(r-s)A} h\|_{L^2}^{2}.
\end{align*} 
Therefore, one has for all $t\ge s$
\begin{align*}
\|\tilde{\eta}_M^h(t|s)\|_{L^2}^2 
& \leq 
\int_s^t \left( 
    \big( \epsilon \| \nabla X_M^x(r)\|_{L^2}^2 +  \tfrac{ 4 C^4}{\epsilon^2} \big) \| \tilde{\eta}^h_M(r|s)\|_{L^2}^2
    +   4   \|X_M^x(r)\|_{L^\infty}^2 \| e^{(r-s) A} h \|_{L^2}^2 
\right) \,dr.
\end{align*}
Recall that $\tilde{\eta}_M^h(s|s)=0$. Appyling Gr\"onwall's lemma and the smoothing property Lemma~\ref{lem:A_analytic}(1),  for all $\alpha \in [0, \frac{1}{2})$, there exists
$C_{\alpha}\in (0,\infty)$ such that for all $s\le t\le T$ one has
\begin{align}	\label{eq:firstder_1}
\|\tilde{\eta}_M^h&(t|s)\|_{L^2}^2 \leq 4 \int_s^t \! \exp\Big(  \epsilon \! \int_{r_1}^t \!  \|\nabla X_M^x(r_2)\|_{L^2}^2  \,dr_2 
+  \tfrac{ 4 C^4}{\epsilon^2} (t- r_1) \Big) \|X_M^x(r_1)\|_{L^\infty}^2 
\| e^{(r_1-s)A} h\|_{L^2}^2 \,dr_1 \nonumber \\
&\leq  C_{\alpha} \exp\Big( \int_s^t \!\!  \epsilon \|\nabla X_M^x(r)\|_{L^2}^2  \,dr +  \tfrac{ 4 C^4}{\epsilon^2}  (t-s) \Big) \underset{r\in [s,t]}\sup~\|X_M^x(r)\|_{L^\infty}^2  \int_s^T (r-s)^{-2\alpha} \,dr \|(-A)^{-\alpha} h\|_{L^2}^2 \nonumber \\
&\leq    C_{\alpha}   T^{1-2\alpha}  \exp(  \tfrac{ 4 C^4}{\epsilon^2}T) 
\exp\Big( \int_s^t \epsilon  \|\nabla X_M^x(r)\|_{L^2}^2  \,dr \Big) \underset{r\in [s,t]}\sup~\|X_M^x(r)\|_{L^\infty}^2 \|(-A)^{-\alpha} h\|_{L^2}^2.
\end{align}
Recalling from~\eqref{eq:pi_etas} that $\Pi_M(t,s) h = \eta_M^h(t|s)=\tilde{\eta}_M^h(t|s) + e^{(t-s) A}h$ for all $t\ge s$, and applying Lemma~\ref{lem:A_analytic}(1) then yields the inequality~\eqref{upperPiM-1},  which concludes the proof of Lemma~\ref{lem_firstder_1}. 
\end{proof} 

Next, one needs to extend the range of admissible values from $\alpha\in[0,\frac12)$ to $\alpha \in [0,1)$. The limitation to $\alpha\in[0,\frac12)$ in Lemma~\ref{lem_firstder_1} is due to using energy inequalities techniques. It is overcome below using the mild formulation~\eqref{eq:VoC_tildeeta} instead. An additional auxiliary positive parameter denoted $\delta\in (0,\infty)$ is required; 
 this parameter may  be chosen arbitrarily small, however, the implied constant $C_{\alpha,\delta,\epsilon}(T)$ may blow up as $\delta$ tends to $0$.
\begin{lemma} \label{lem_firstder_2}
Let $\alpha \in [ 0, 1)$, and $\delta,\epsilon \in (0,\infty)$. There exists  $C_{\alpha,\delta,\epsilon}(T) \in (0,\infty)$ such that for all $M\in \N$, $x\in H_M$, all $h\in H_M\setminus\{0\}$ and $0\leq s < t \leq T$ one has
 \begin{align}	\label{upperPiM-2}
 \frac{\| \Pi_M(t,s)  h \|_{L^2}}{\| (-A)^{-\alpha} h \|_{L^2}}
 \leq  &\; 
 \frac{C_{\alpha,\delta,\epsilon}(T)}{(t-s)^{\alpha}} e^{ \epsilon\! \int_s^t \!\|\nabla  X_M^x(r)\|_{L^2}^2 \,dr} 
\Big( \underset{r\in [s,t]}\sup~ \|(-A)^{\frac{1}{4}+\delta}X_M^x(r)\|_{L^2}^{2}+1\Big).
\end{align}
\end{lemma}

\begin{proof}
In view of Lemma~\ref{lem_firstder_1} (and of the Sobolev inequality~\eqref{eq:Linfty_DA_bound}), we only need to consider the case $\alpha \in [\frac{1}{2},1)$. 
Without loss of generality in this proof assume that $\delta\in(0,1-\alpha)$. Let also $M\in\N$, $T\in(0,\infty)$ and $x,h\in H_M$ be fixed.

Using the expression~\eqref{eq:VoC_tildeeta} of $\tilde{\eta}_M^h(t|s)$ and applying Lemma~\ref{lem_firstder_1} 
(with $\alpha=\frac12-\frac{\delta}{4}$), there exists $C_{\delta,\epsilon}(T)\in (0,\infty)$  such that for all $t\in[s,T]$ one has
\begin{align}\label{eq:tilde_eta_est_2}
\| \tilde{\eta}^h_M(t|s)\|_{L^2} 
& \leq 
2 \int_s^t \| \Pi_M(t,r) B_M[X_M^x(r), e^{(r-s)A} h] \|_{L^2} \,dr 
\notag \\
& 
\leq C_{\delta,\epsilon}(T) 
e^{ \epsilon\! \int_s^t \!\|\nabla  X_M^x(r)\|_{L^2}^2 \,dr} \Big( \underset{r\in [s,t]}\sup~\|X_M^x(r)\|_{L^\infty} +1\Big)
\notag \\ &
\quad \times 
\int_{s}^{t}  (t-r)^{-\frac{1}{2}+\frac{\delta}{4}}
\| (-A)^{-\frac{1}{2}+\frac{\delta}{4}} B[X_M^x(r), e^{(r-s)A}h] \|_{L^2} \,dr .
\end{align}
Since  $\alpha\ge \frac12$ and  $\delta<1-\alpha$, one has $\delta<\frac{1}{2}$; therefore Lemmas~\ref{lem:deriv_bdd_fracdomains} 
(with $\alpha=\frac12-\frac{\delta}{4}$) and~\ref{lem:fracpow_B_est} (with $\gamma=\frac{\delta}{4}$) imply that 
there exists  $C_{\delta}\in(0,\infty)$ such that for all $r\in[s,t]$ one has
\begin{align*}
\| (-A)^{-\frac{1}{2}+\frac{\delta}{4}} B[X_M^x(r), e^{(r-s)A}h] \|_{L^2} &= \| (-A)^{-\frac{1}{2}+\frac{\delta}{4}}  \nabla \big(X_M^x(r) e^{(r-s)A}h\big)  \|_{L^2} \\
& \leq  \| (-A)^{\frac{\delta}{4}} \big(X_M^x(r) e^{(r-s)A}h\big)  \|_{L^2} \\
&\leq C_{\delta} \|(-A)^{\frac{1}{4}+\delta} X_M^x(r)\|_{L^{2}} \|(-A)^{\delta} e^{(r-s) A}h \|_{L^{2}}.
\end{align*}
Plugging this upper bound into~\eqref{eq:tilde_eta_est_2}, and applying the smoothing property from Lemma \ref{lem:A_analytic}(1) and the Sobolev inequality~\eqref{eq:Linfty_DA_bound},
one deduces that there exists $C_{\alpha,\delta,\epsilon}(T)\in(0,\infty)$ such that for all $t\in[s,T]$ one has
\begin{align*}
\| \tilde{\eta}^h_M(t|s ) \|_{L^2} 
&
\leq C_{\alpha,\delta,\epsilon}(T) e^{ \epsilon\! \int_s^t \!\|\nabla  X_M^x(r)\|_{L^2}^2 \,dr}  
\big(\underset{r\in [s,t]}\sup~\|(-A)^{\frac{1}{4}+\delta } X_M^x(r)\|_{L^2}^{2} +1\big)
\\ & \quad 
\times \int_s^t (t-r)^{-\frac{1}{2} +\frac{\delta}{4}} (r-s)^{-\delta  -\alpha} \,dr \;   \|(-A)^{-\alpha} h\|_{L^2} .
\end{align*}
Owing to the condition $\delta +\alpha <1$ imposed at the beginning of the proof on the parameter $\delta$,  the change of variable $r=s+(t-s)\theta$ in the integral above implies that there exists $C_{\alpha,\delta}(T)\in (0,\infty)$ such that 
\[
\int_s^t\!\! (t-r)^{-\frac{1}{2} +\frac{\delta}{4}} (r-s)^{-\delta -\alpha} \,dr = (t-s)^{\frac{1}{2} - \frac{3\delta}{4} -\alpha } \int_0^1 \!\!\theta^{-(\alpha + \delta)} (1-\theta)^{-\frac{1}{2} +\frac{\delta}{4}}  d\theta 
\leq  C_{\alpha,\delta}(T) (t-s)^{-\alpha},
\]
where the last upper bound uses the inequality $(t-s)^{\frac12-\frac{3\delta}{4}}\le T^{\frac12-\frac{3\delta}{4}}$ which is deduced  from  the condition $\delta <1-\alpha\le \frac{1}{2}$.

Recalling that  from~\eqref{eq:pi_etas}  one has $\Pi_M(t,s) h = \tilde{\eta}_M^h(t|s) + e^{(t-s)A} h$ for all $t\ge s$, and applying  Lemma \ref{lem:A_analytic}(1), 
this yields the inequality~\eqref{upperPiM-2}, which concludes the proof of Lemma~\ref{lem_firstder_2}.
\end{proof}

It remains to state and prove the last auxiliary result on the linear operators $\Pi_M(t,s)$.

\begin{lemma} \label{lem_firstder_3}
Let $\alpha \in [0, 1)$, $\gamma \in [0,\frac{1}{2})$, and $\delta, \epsilon \in (0,\infty)$. There exists  $C_{\alpha,\gamma,\delta,\epsilon}(T) \in (0,\infty)$ such that for all $M\in N$, $x\in H_M$, $h\in H_M\setminus \{0\}$,
and  $0\leq s < t\leq T$,  one has
 \begin{align}	\label{upperPiM-3}
\frac{\| (-A)^\gamma \Pi_M(t,s) h \|_{L^2}}{\| (-A)^{-\alpha} h \|_{L^2}} & 
\leq  \frac{C_{\alpha,\gamma,\delta,\epsilon}(T)}{(t-s)^{\alpha+\gamma}}e^{ \epsilon\! \int_s^t \!\|\nabla  X_M^x(r)\|_{L^2}^2 \,dr} \Big( \underset{r\in [s,t]}\sup~ \|(-A)^{\frac{1}{4}+\delta}X_M^x(r)\|_{L^2}^{3}+1\Big).
\end{align}
\end{lemma}

\begin{proof}
Let $M\in\N$, $T\in(0,\infty)$ and $x,h\in H_M$ be fixed.

Recall that $\Pi_M(t,s)h = \eta_M^h(t|s)$,  and that the mild formulation for the solution $\eta_M^h(t|s)$ to~\eqref{eta(t|s)}  is given  by~\eqref{eq:eta_mild}.

Therefore,  applying Lemma~\ref{lem:A_analytic}(1) (with $\alpha=\frac12+\gamma$) and Lemma~\ref{lem:deriv_bdd_fracdomains}, then Lemma~\ref{lem_firstder_2} and the Sobolev 
inequality~\eqref{eq:Linfty_DA_bound},  one deduces that there exists $C_{\alpha+\gamma}, C_{\gamma,\delta,\epsilon}(T), C_{\alpha,\gamma,\delta,\epsilon}(T) \in (0,\infty)$ 
 such that for all $t\ge s$ one has
\begin{align*}
\| (-A)^\gamma \eta_M^h(t|s) \|_{L^2} 
& \leq 
\| (-A)^\gamma e^{(t-s) A} h\|_{L^2} + C_{\gamma,\delta,\epsilon}(T)  \int_s^t (t-r)^{-(\frac{1}{2} +\gamma)} \|X_M^x(r)\|_{L^\infty} \|\eta_M^h(r|s)\|_{L^2}\, dr\\
& \leq   C_{\alpha+\gamma} (t-s)^{-(\alpha+\gamma)} \|(-A)^{-\alpha} h\|_{L^2} \\
&\quad + C_{\alpha,\gamma,\delta,\epsilon}(T) \int_s^t 
(t-r)^{-(\frac{1}{2}+\gamma)} (r-s)^{-\alpha}\,  dr  \\
&\qquad \times
e^{ \epsilon\! \int_s^t \!\|\nabla  X_M^x(r)\|_{L^2}^2 \,dr}
\Big( \underset{r\in [s,t]}\sup~\|(-A)^{\frac{1}{4}+\delta }X_M^x(r)\|_{L^2}^3 +1\Big)
\| (-A)^{-\alpha}h\|_{L^2}.
\end{align*}
Owing to the conditions $\frac{1}{2}+\gamma <1$ and $\alpha <1$, performing the change of variable $r=s+(t-s)\theta$ in the integral above, one has
\[
\int_s^t (t-r)^{-(\frac{1}{2}+\gamma)} (r-s)^{-\alpha}\,  dr \leq T^{\frac{1}{2}}  (t-s)^{-\alpha- \gamma }\int_0^1 (1-\theta)^{-(\frac{1}{2}+\gamma)} \theta^{-\alpha} \,  d\theta.
\]
One thus obtains the inequality~\eqref{upperPiM-3}, which concludes the proof of Lemma~\ref{lem_firstder_3}.
\end{proof}

We are now in a position to state and
  prove moment bounds for $\eta_M^h(t)$ and $\zeta_M^{g,h}(t)$ in Corollary~\ref{cor_moments_Du} below , using the properties of the linear operators $\Pi_M(t,s)$ obtained above. 
 Corollary~\ref{cor_moments_Du} is the final ingredient needed for the proof of Theorem~\ref{theo:Kol_reg}.

\begin{cor}	\label{cor_moments_Du}
\begin{enumerate}
\item\label{cor_moments_Du1}  Let $\alpha \in [0,1)$, $\gamma \in [0, \frac{1}{2})$,  $\delta, \epsilon \in (0,\infty)$; then there exists  
$C_{p,\alpha,\gamma,\delta,\epsilon}(T,Q) \in (0,\infty) $  such that for all $M\in\N$, $x,h\in H_M$, and $t\in [0,T]$, one has
\begin{equation}	\label{mom_D1}
  \bigl(\E\bigl[\|(-A)^{\gamma} \eta_M^h(t)\|_{L^2}^p \bigr]\bigr)^{\frac{1}{p}} 
  \leq C_{p,\alpha,\gamma,\delta,\epsilon}(T,Q) t^{  - (\alpha + \gamma)} e^{\epsilon\|  x\|_{L^2}^2} 
  \bigl( 1+\|(-A)^{\frac{1}{4}+\delta} x\|_{L^2}^{6}\bigr)
 \|(-A)^{-\alpha}h\|_{L^2},
 \end{equation}
 where we recall that $t\mapsto \eta_M^{h}(t)$ is the solution 
to~\eqref{eq:X_deriv_initialvalue}.
 \item\label{cor_moments_Du2} Let $\alpha, \beta \in [0,1)$ be such that $\alpha+\beta <1$, and let $\delta , \epsilon \in (0,\infty)$.   
 There exists  $ C_{p,\alpha,\beta,\delta,\epsilon, \mu }(T,Q)\in (0,\infty) $ such that for all 
 $M\in\N$, $x,g,h\in H_M$, $t\in [0,T]$, one has
 \begin{equation}	\label{mom_D2}
 \begin{aligned}
 & \bigl(\E\bigl[ \| \zeta_M^{g,h}(t)\|_{L^2}^p \bigr]\bigr)^{\frac{1}{p}} 
 \\ & \qquad \leq C_{p,\alpha,\beta,\delta,\epsilon}(T,Q) t^{-(\alpha + \beta)} e^{\epsilon \|x\|_{L^2}^2}
  \bigl( 1+\|(-A)^{\frac{1}{4}+\delta} x\|_{L^2}^{16}\bigr) \|(-A)^{-\alpha}g\|_{L^2} \| (-A)^{-\beta}h\|_{L^2}, 
 \end{aligned}
 \end{equation}
 where we recall that $t\mapsto \zeta_M^{g,h}(t)$ is the solution  to~\eqref{eq:X_2nd_deriv_initialvalue}.
 \end{enumerate}
\end{cor}

\begin{proof} 
\textbf{Proof of \ref{cor_moments_Du1}.}
Let us first prove the inequality~\eqref{mom_D1}. Recall that $\eta_M^h(t)=\Pi_M(t,0)h$. Therefore, the inequality~\eqref{upperPiM-3} from Lemma~\ref{lem_firstder_3} (applied with $\epsilon$ replaced by $\epsilon/2$) and the Cauchy-Schwarz inequality imply,  given $\epsilon \in (0, \infty)$, the existence of $C_{\alpha,\gamma,\delta,\epsilon}(T) \in (0,\infty)$ such that for all $M\in\N$, $x,h\in H_M$, $t\in [0,T]$.
\begin{align*}
 \bigl( \E\bigl[ \|(-A)^{\gamma} \eta_M^h(t)\|_{L^2}^p \bigr]\bigr)^{\frac{1}{p}} 
 & \leq  C_{\alpha,\beta,\delta,\epsilon}(T)  t^{ -(\alpha + \gamma)} \|(-A)^{-\alpha}h\|_{L^2} 
 \Bigl( \E \Bigl[ e^{ \epsilon p\int_0^t \!\|\nabla  X_M^x(r)\|_{L^2}^2 \,dr}\Bigr]\Bigr)^{\frac{1}{2p}} \\
 &\quad \times  
 \Bigl( \E\Bigl[1+ \underset{s\in [0,T]}\sup~ \|(-A)^{\frac{1}{4}+ \frac{\delta}{2}}  X_M^x(s) \|^{ 6 p }_{L^2} \Bigr]\Bigr)^{\frac{1}{2p}}.
\end{align*}

The exponential moment bounds~\eqref{expmoments_detIV} from Lemma~\ref{lem:expest}  (applied with $\lambda=\frac14+\frac{\delta}{2}$ and $\gamma=\frac{\delta}{2}$) imply that 
for all $\epsilon \in (0,\frac{1}{4p \|Q\|_{\mathcal{L}(L^2)}})$  there exists $C_{\epsilon,p}(T,Q) \in (0,\infty)$ such that
\[
\underset{M\in\N}\sup~ \Bigl( \E \Bigl[ e^{\frac{\epsilon}{2} p\int_0^t \!\|\nabla  X_M^x(r)\|_{L^2}^2 \,dr }\Bigr]\Bigr)^{\frac{1}{2p}}  \leq C_{\epsilon,p}(T,Q) 
e^{\epsilon \|x\|_{L^2}^2}. 
\]

The regularity result given by inequality~\eqref{eq:XM_moment_mixed_reg} from Lemma~\ref{lem:X_sup_moments} implies that there exists 
$C_{\delta,p, \mu}(T,Q)\in (0,\infty)$ such that
\begin{align*}
 \underset{M\in\N}\sup~
    \Bigl( 
        \E \Bigl[ 
            \underset{s\in [0,T]}\sup~ 
                \|(-A)^{\frac{1}{4}+\frac{\delta}{2}} X_M^x(s) \|_{L^2}^{6p}
        \Bigr] 
    \Bigr)^{\frac{1}{2 p}}
  &
  \leq C_{\delta,p}(T,Q) \bigl(1+ \|(-A)^{\frac{1}{4}+\delta } x\|_{L^2}^{6}  + \|x\|_{L^2}^{18}  \bigr).
\end{align*} 
 Since $\|x\|_{L^2}^{18} \leq C_{ \epsilon} e^{\frac{\epsilon}{2} \|x\|_{L^2}^2}$, combining the upper bounds above then yields the inequality~\eqref{mom_D1}.
\smallskip

\textbf{Proof of \ref{cor_moments_Du2}.} Let us now prove the inequality~\eqref{mom_D2}. Recall the expression of $\zeta_M^{g,h}(t)$ given by~\eqref{eq:VOC_zeta}.

Let 
 $\lambda \in \bigl( 0\vee   (\alpha + \beta - \frac{3}{4}), \frac{1}{4}\bigr)$
denote an auxiliary parameter. 
 Applying the inequality~\eqref{upperPiM-2} in Lemma~\ref{lem_firstder_2} (with $\alpha=\frac34+\lambda$) 
implies that there exists $C_{\lambda,\delta,\epsilon}(T)\in (0,\infty)$ such that, for all $M\in \N$, $x,h\in H_M$, and $t\in [0,T]$, one has
\begin{align}\label{eq:first_zetaM_est}
\| \zeta_M^{g,h}(t) \|_{L^2} 
& \leq 
C_{\lambda,\delta,\epsilon}(T) e^{ \frac{\epsilon}{2}\! \int_0^t \!\|\nabla  X_M^x(r)\|_{L^2}^2 \,dr}  
\Big( \underset{r\in [0,t]}\sup~ \|(-A)^{\frac{1}{4} + \frac{\delta}{2}} X_M^x(r)\|_{L^2}^2+1\Big) \nonumber \\ & \quad 
\times \int_0^t (t-s)^{-(\frac{3}{4}+\lambda)} 
\| (-A)^{-(\frac{3}{4}+\lambda)} 
B_M[\eta^g_M(s), \eta_M^h(s)] \|_{L^2} \,ds.
\end{align}
Applying the inequality~\eqref{eq:Linfty_DA_bound-dual}, one obtains the upper bound
\begin{align*}
 \| (-A)^{-(\frac{3}{4}+\lambda)} B_M[\eta^g_M(s), \eta_M^h(s)] \|_{L^2}
 & \leq 
 \| (-A)^{-(\frac{3}{4}+\lambda)} B[\eta^g_M(s), \eta_M^h(s)] \|_{L^2}
\\ & \leq C_{\lambda} \| (-A)^{-\frac{1}{2}} B[\eta^g_M(s) \eta_M^h(s)] \|_{L^{1}}.
\end{align*}
Set $q = \frac{2}{1+4\lambda}$ and observe that by definition of $\lambda$ we have $q>1$. Applying the inequality~\eqref{eq:deriv_bdd_sqrtA} from Lemma~\ref{lem:deriv_bdd_sqrtA} and 
H\"older's inequality implies that there exists  $C_{q}\in(0,\infty)$ such that one has the upper bounds
\begin{align*}
\| (-A)^{-\frac{1}{2}} B[\eta^g_M(s), \eta_M^h(s)] \|_{L^{1}}
& \leq
\| (-A)^{-\frac{1}{2}} B[\eta^g_M(s), \eta_M^h(s)] \|_{L^{q}} \\
& \leq 
C_{q}
\| \eta^g_M(s) \eta_M^h(s) \|_{L^{q}}
\\ & 
\leq C_{q} \| \eta^{g}_M(s) \|_{L^{2q}}\| \eta^{h}_M(s) \|_{L^{2q}}.
\end{align*}
From the Sobolev inequality~\eqref{eq:SobolevI} and Proposition~\ref{prop:DA_Sobolev_equiv}, one then obtains the upper bound
\begin{align*}\| (-A)^{-\frac{1}{2}} B[\eta^g_M(s) ,\eta_M^h(s)] \|_{L^{1}}
& \leq 
C_{q} \| (-A)^{\frac{1}{4}-\frac{1}{4q}}\eta^{g}_M(s) \|_{L^{2}}
\| (-A)^{\frac{1}{4}-\frac{1}{4q}}\eta^{h}_M(s) \|_{L^{2}}.
\end{align*}
Applying 
the inequality~\eqref{upperPiM-3} from Lemma~\ref{lem_firstder_3} with $\gamma=\frac{1}{4}-\frac{1}{4q}=\frac{1}{8}-\frac{\lambda}{2}$, 
one deduces that there exists $C_{\alpha,\beta,\lambda,\delta,\epsilon}(T)\in(0,\infty)$ such that for all $M\in \N$, all $x,g,h\in H_M$, and all $s\in [0,T]$
\begin{align*}
\| (-A)^{-\frac{1}{2}} B[\eta^g_M(s), \eta_M^h(s)]\|_{L^{1}}
& \leq 
C_{\alpha,\beta,\lambda,\delta,\epsilon}(T) s^{  -(\alpha +\beta + \frac{1}{4}-\lambda)}  e^{ \frac{\epsilon}{2}\! \int_0^s \!\|\nabla  X_M^x(r)\|_{L^2}^2 \,dr}
\\ & \quad \times 
\Big( \underset{r\in [0,s]}\sup~ \|(-A)^{\frac{1}{4} + \frac{\delta}{2} } X_M^x(r)\|_{L^2}^{6} +1\Big) \| (-A)^{-\alpha} h \|_{L^2}\| (-A)^{-\beta} g \|_{L^2}.
\end{align*}
Plugging the above inequality in~\eqref{eq:first_zetaM_est} and recalling that $\lambda$ is an auxiliary parameter chosen such that $\alpha+\beta+\frac14-\lambda<1$, one thus obtains that there exists a $C_{\alpha,\beta,\delta,\epsilon}(T)\in (0,\infty)$ such that
\begin{align*}
\| \zeta_M^{g,h}(t) \|_{L^2} 
& \leq 
C_{\alpha,\beta,\delta,\epsilon}(T) e^{ \epsilon\! \int_0^t \!\|\nabla  X_M^x(r)\|_{L^2}^2 \,dr}  
\Big( \underset{r\in [0,t]}\sup~\|(-A)^{\frac{1}{4} + \frac{\delta}{2} } X_M^x(r)\|_{L^2}^{8} +1\Big) \\ & \quad 
\times \int_0^t (t-s)^{ - (\frac{3}{4} + \lambda)  } s^{ -(\alpha+\beta + \frac{1}{4} - \lambda) } \,ds
\| (-A)^{-\alpha} h \|_{L^2}\| (-A)^{-\beta} g \|_{L^2}
\\ & \leq 
 C_{\alpha,\beta,\delta,\epsilon}(T) t^{ -(\alpha+\beta)} e^{ \epsilon\! \int_0^t \!\|\nabla  X_M^x(r)\|_{L^2}^2 \,dr}  
\Big( \underset{r\in [s,t]}\sup~ \|(-A)^{\frac{1}{4} + \frac{\delta}{2} } X_M^x(r)\|_{L^2}^{8} +1\Big) \\ & \quad 
\times 
\| (-A)^{-\alpha} h \|_{L^2}\| (-A)^{-\beta} g \|_{L^2}.
\end{align*} 
In order to obtain the inequality~\eqref{mom_D2}, it thus remains to proceed as in the proof of the inequality~\eqref{mom_D1}, using H\"older's inequality, the  
exponential moment bounds \eqref{expmoments_detIV}  from Lemma~\ref{lem:expest} and the regularity properties from Lemma~\ref{lem:X_sup_moments}. 
The details are omitted. This concludes the proof of the inequality~\eqref{cor_moments_Du2}.

The proof of Corollary~\ref{cor_moments_Du} is thus completed.
\end{proof}

\begin{proof}[Proof of Theorem~\ref{theo:Kol_reg}]
Recall that the mapping $\varphi:L^2\to\R$ is assumed to be twice continuously differentiable with bounded first and second order derivatives. Moreover, the expressions of $D u_M(t,x). (h)$ 
and $D^2 u_M(t,x) .(g,h)$ are given by~\eqref{eq:Kol_deriv} and~\eqref{eq:Kol_2nd_deriv} respectively.

To prove the inequality~\eqref{eq:Kol_deriv_est}, it suffices to note that owing to~\eqref{eq:Kol_deriv} one has
\[
|D u_M(t,x) . (h)|\le C(\varphi)\E[\|\eta_M^h(t)\|_{L^2}]
\]
and to apply the inequality~\eqref{mom_D1} from Corollary~\ref{cor_moments_Du}.

To prove the inequality~\eqref{eq:Kol_2nd_deriv_est}, observe that owing to~\eqref{eq:Kol_2nd_deriv} one has
\[
|D^2 u_M(t,x).(g,h)|\le C(\varphi)\Bigl(\bigl(\E[\|\eta_M^g(t)\|_{L^2}^2]\bigr)^{\frac12} \bigl(\E[\|\eta_M^h(t)\|_{L^2}^2]\bigr)^{\frac12} +\E[\|\zeta_M^{g,h}(t)\|_{L^2}]\Bigr).
\]
Applying the inequalities~\eqref{mom_D1} and~\eqref{mom_D2} from Corollary~\ref{cor_moments_Du} then yields the inequality~\eqref{eq:Kol_2nd_deriv_est}.

The proof of Theorem~\ref{theo:Kol_reg} is thus completed.
\end{proof}

\begin{rem}
    The degrees $1$, $2$ and $3$ of polynomial dependence with respect to $X_M^x$ in the inequalities~\eqref{upperPiM-1} from Lemma~\ref{lem_firstder_1},~\eqref{upperPiM-2} from Lemma~\ref{lem_firstder_2} and~\eqref{upperPiM-3} from Lemma~\ref{lem_firstder_3} may not be optimal. Similarly, the degrees $6$ and $16$ for the dependence with respect to $\|(-A)^{\frac{1}{4}+\delta } x\|_{L^2}$ appearing in Corollary~\ref{cor_moments_Du} and finally in Theorem~\ref{theo:Kol_reg} may also not be optimal. See also Remark~\ref{rem:notoptimal}.
\end{rem}

\section{Weak error estimates for the spectral Galerkin approximation}\label{weak-rate}

This section is devoted to the rigorous statement of the main result of this article: we provide weak error estimates for the spectral Galerkin approximation.

\begin{theo}\label{thm:weakrate}
Assume that there exists $\gamma_0\in(0,\infty)$ such that $ \E [\exp(\gamma_0 \| X_0\|_{L^2}^2 ) ]< \infty$, and that there exist $p\in (32,\infty)$
and $\delta_0\in(0,\infty)$ such that $\E\bigl[\| (-A)^{\frac14 + \delta_0} X_0 \|_{L^2}^{p}\bigr]<\infty$.

Let $X$ be the solution to the Burgers equation~\eqref{eq:Burgers}, and let $X_N$, $N\in \N$, be the associated Galerkin approximations~\eqref{eq:GalsolBurgers}. Let $\varphi\in C^2(L^2,\R)$ have bounded first and second derivatives. For all $\alpha\in (0,1 \wedge (\frac{3}{4}+\delta_0))$ 
there exists $C_{\alpha,\gamma_0,\delta_0,p}(T,Q,\varphi)\in (0,\infty)$ such that for all $N\in\N$ one has
\begin{equation}
\begin{aligned}
&    \left|\E[ \varphi(X(T))] - \E[\varphi(X_N(T))]\right|
\\ & \qquad    \leq 
    C_{\alpha,\gamma_0,\delta_0,p}(T,Q,\varphi) N^{-2\alpha}\left(1+\E\bigl[\exp\bigl(\gamma_0\|X_0\|_{L^2}^2\bigr)\bigr] \right)^2\left(1+\E\bigl[\|-A)^{\frac14+\delta_0}X_0\|_{L^2}^{p}\bigr]\right).
\end{aligned}
\end{equation}
\end{theo}

\begin{rem}\label{rem:weak_vs_strong}
Before providing the proof of Theorem~\ref{thm:weakrate}, let us compare the weak error estimates with strong error bounds obtained previously in the literature.
We recall from~\cite[Equation (107)]{HutzenthalerJentzen:2020} that for all $p\in [8,\infty)$,  all $r,q\in (0,\infty)$ such that 
$\frac{4}{p}+\frac{1}{q}=\frac{1}{r}$ and all $\alpha\in (\frac{1}{2},\infty)$,  there exists $C_{\alpha,\gamma_0,p,q}(T,Q)\in(0,\infty)$ such that 
\begin{equation*}\label{eq:strongGalConvRates_pre}
\begin{aligned}
\sup_{M\in \N} \left( 
M^{\alpha} \sup_{t\in [0,T]} \left(\E [\| X(t) - X_M(t) \|_{L^2}^r ] \right)^{\frac{1}{r}}
\right)
&\leq 
C_{\alpha,\gamma_0,p,q}(T,Q) \left( \E[\exp(\gamma_0 \| X_0 \|_{L^2}^2)]\right)^{\frac{1}{q}} 
\\ & \quad \times 
\left( 
1 \wedge \liminf_{M\rightarrow \infty} \sup_{t\in [0,T]} 
\left( \E[\| (-A)^{\frac{\alpha}{2}}X_M(t) \|_{L^2}^{\frac{p}{2}}]\right)^{\frac{4}{p}}
\right).
\end{aligned}
\end{equation*}
In view of Lemma~\ref{lem:X_sup_moments} this implies that under the assumptions of Theorem~\ref{thm:weakrate} one has for all $r,q\in (0,\infty)$ satisfying $\frac{4}{p}+\frac{1}{q}=\frac{1}{r}$ and all $\alpha\in (0,1 \wedge (\frac12+2\delta_0))$ that there exists a $C_{\alpha,\gamma_0,\delta_0,p,q}(T,Q)\in (0,\infty)$ such that for all $N\in\N$ one has
\begin{equation}\label{eq:strongGalConvRates}
\begin{aligned}
&
\sup_{t\in [0,T]} \left(\E [\| X(t) - X_N(t) \|_{L^2}^r ] \right)^{\frac{1}{r}}\le 
C_{\alpha,\gamma_0,\delta_0,p,q}(T,Q,X_0)
N^{-\alpha},
\end{aligned}
\end{equation}
where
\begin{align*}
&C_{\alpha,\gamma_0,\delta_0,p,q}(T,Q,X_0)\\
& \quad \le C_{\alpha,\gamma_0,\delta_0,p,q}(T,Q)\left( \E[\exp(\gamma_0 \| X_0 \|_{L^2}^2)]\right)^{\frac{1}{q}}\left( 1 + \E\bigl[\| (-A)^{\frac14+\delta_0} X_0 \|_{L^2}^{p}\bigr] +  \E\bigl[\| X_0 \|_{L^2}^{ 3p}\bigr]\right)^{\frac{4}{p}}<\infty.
\end{align*}
Note that $ \E [\exp(\gamma_0 \| X_0\|_{L^2}^2 ) ]< \infty$ implies $\E\bigl[\| X_0 \|_{L^2}^{ 3p}\bigr]<\infty$.

Observe that when $\delta_0\ge \frac14$, one may take arbitrary $\alpha\in(0,1)$, meaning that the weak order $2\alpha$ is twice the strong order $\alpha$. For $\delta_0 \in [0,\frac{1}{4})$ the weak rate we obtain is \emph{more than twice} the strong rate obtained in~\cite{HutzenthalerJentzen:2020}. However, we have no indication that the (weak or strong) rates obtained in the regime where $\delta_0 \in (0,\frac{1}{4})$ are optimal. In particular,~\cite[Corollary 7.5]{ConusJentzenKurniawan:2019} proves that a weak rate $2$ is optimal for Galerkin approximations of certain linear parabolic equations. Moverover, results in~\cite{MullerGronbachRitter:2007} imply that a strong rate $1$ is optimal for Galerkin approximations of certain semilinear parabolic equations. We do not know of any lower bounds for approximations of the Burgers' equation (weak or strong), let alone that it is known whether they depend on the regularity of the initial value. 
\end{rem}

In order to prove Theorem~\ref{thm:weakrate}, two additional auxiliary lemmas are required. First, Lemma~\ref{lemma:Kolmogorov_randomIV} gives a property on the solutions $u_M$ of the Kolmogorov equation, in order to extend the definition~\eqref{eq:Kolmogorov} of $u_M(t,x)$ for deterministic $x\in H_M$ to random initial values $x=X_M(0)$. The result is classical but requires some care to deal with the nonlinearity in a rigorous way. A proof is provided in Appendix~\ref{app:Kolmogorov}.

\begin{lemma}\label{lemma:Kolmogorov_randomIV}
Assume that there exists $p\in [4,\infty)$ such that $\E[\| X_0 \|_{L^2}^p]<\infty$. Let $M\in\N$ and let $\varphi \colon L^2 \rightarrow \R$ be twice continuously differentiable with bounded first and second derivatives, and let $u_M$ be defined by~\eqref{eq:Kolmogorov}. Then one has
\begin{equation}\label{eq:Kolmogorov_randomIV}
u_M(t,X_M(0)) = \E[ \varphi( X_M(t) ) | X_M(0) ] \quad \text{a.s.}
\end{equation}
for all $t\in [0,T]$, $M\in \N$.
\end{lemma}

Second, given parameters $\delta,\epsilon\in(0,\infty)$ and $q\in[1,\infty)$, set
\begin{equation}\label{eq:defPsi}
\Psi_{\delta,\epsilon,q}(x)=\exp\bigl(\epsilon \|x\|_{L^2}^2\bigr)  \bigl( 1+\|(-A)^{\frac{1}{4}+\delta } x\|_{L^2}^{q}\bigr),\quad \forall~x\in D((-A)^{\frac14+\delta}).
\end{equation}

\begin{lemma}\label{lem:Psi}
Assume that there exists $\gamma_0\in(0,\infty)$ such that $ \E [\exp(\gamma_0 \| X_0\|_{L^2}^2 ) ]< \infty$, and that there exist $p\in (32,\infty)$ and $\delta_0\in(0,\infty)$ such that $\E\bigl[\| (-A)^{\frac14 + \delta_0} X_0 \|_{L^2}^{p}\bigr]<\infty$.

Then for all $q\in[1,\frac{p}{2})$, $\delta\in[0,\delta_0)$ and $\epsilon\in(0,(1-\tfrac{2q}{p})\frac{\gamma_0}{1+2\gamma_0\|Q\|_{\mathcal{L}(L^2)}})$, there exists $C_{\gamma_0,\delta,\delta_0,\epsilon,p,q}(T,Q)\in(0,\infty)$ such that
\begin{equation}\label{eq:lemPsi}
\begin{aligned}
\underset{M\in\N}\sup~&\E\Big[\underset{t\in[0,T]}\sup \Psi_{\delta,\epsilon,q}(X_M(t))\Big]\\
&\le C_{\gamma_0,\delta,\delta_0,\epsilon,p,q}(T,Q)\left(1+\E\bigl[\exp\bigl(\gamma_0\|X_0\|_{L^2}^2\bigr)\bigr] \right)^2\left(1+\E\bigl[\|-A)^{\frac14+\delta_0}X_0\|_{L^2}^{p}\bigr]\right)^{\frac{2q}{p}}.
\end{aligned}
\end{equation}
\end{lemma}

\begin{proof}
Recall that $\| (-A)^{\alpha} x \|_{L^2} \leq  \| (-A)^{\beta} x\|_{L^2}$ for all $x\in D((-A)^{\beta})$ and all $\alpha<\beta$
(see~\eqref{eq:A_fracpownorminc}). Thus, without loss of generality, we can assume $\delta_0< \frac{1}{4}$. 
Applying H\"older's inequality, one has for all 
$M\in\N$
\begin{align*}
\E\Big[\sup_{t\in [0,T]} \Psi_{\delta,\epsilon,q}(X_M(t))\Big]\leq &\left( \E\Big[\sup_{t\in [0,T]} \exp\bigl(\tfrac{p\epsilon}{p-2q}\|X_M(t)\|_{L^2}^2 \bigr)\Big]\right)^{1-\frac{2q}{p}}\\
&\times \left( \E\Big[\sup_{t\in [0,T]} \bigl(1+\|(-A)^{\frac{1}{4}+\delta } X_M(t)\|_{L^2}^{q}\bigr)^{\frac{p}{2q}}\Big] \right)^{\frac{2q}{p}}.
\end{align*}
Applying the exponential moment bounds~\eqref{expmoments_ranIV} from Lemma~\ref{lem:expest} (with $\beta=\frac{p\epsilon}{p-2q}$) and the inequality~\eqref{eq:XM_moment_mixed_reg} from Lemma~\ref{lem:X_sup_moments} (with $\lambda=\frac14+\delta$, $\gamma \in (0,\delta_0-\delta)$, $\alpha=\frac14+\delta_0$) one obtains that there exists $C_{\epsilon,p,q},\,C_{\delta,\delta_0,p}\in (0,\infty)$ such that
\begin{align*}
&\E\Big[\sup_{t\in [0,T]} \exp\bigl(\tfrac{p\epsilon}{p-2q}\|X_M(t)\|_{L^2}^2 \bigr)\Big]\le C_{\epsilon,p,q}(T,Q)\left( \E\bigl[\exp\bigl(\gamma_0 \|X_0\|_{L^2}^2 \bigr)\bigr]\right)^{\frac{p\epsilon}{(p-2q)\gamma_0}}\\
&\E\Big[ \sup_{t\in [0,T]} \bigl(1+\|(-A)^{\frac{1}{4}+\delta } X_M(t)\|_{L^2}^{q}\bigr)^{\frac{p}{2q}}\Big]\le C_{\delta,\delta_0,p}(T,Q)\Bigl( 1+\E \bigl[\| (-A)^{\frac14+\delta_0} X_0 \|_{L^2}^{p}\bigr]+  \E \bigl[\| X_0 \|_{L^2}^{3p} \bigr]\Bigr).
\end{align*}
Note that $\epsilon<\gamma_0$, therefore one has
\[
\left( \E\bigl[\exp\bigl(\gamma_0 \|X_0\|_{L^2}^2 \bigr)\bigr]\right)^{\frac{\epsilon}{\gamma_0}}\le 1+\E\bigl[\exp\bigl(\gamma_0 \|X_0\|_{L^2}^2 \bigr)\bigr].
\]
Moreover, there exists $C_{\gamma_0,p}\in(0,\infty)$ such that
\begin{align*}
1+\E \bigl[\| (-A)^{\frac{1}{4}+\delta_0} X_0 \|_{L^2}^{p}\bigr]+  \E \bigl[\| X_0 \|_{L^2}^{3p} \bigr]&\le \bigl(1+\E \bigl[\| (-A)^{\frac{1}{4}+\delta_0} X_0 \|_{L^2}^{p}\bigr]\bigr)\bigl(1+\E \bigl[\| X_0 \|_{L^2}^{3p} \bigr]\bigr)\\
&\le C_{\gamma_0,p}\bigl(1+\E \bigl[\| (-A)^{\frac{1}{4}+\delta_0} X_0 \|_{L^2}^{p}\bigr]\bigr)\bigl(1+\E\bigl[\exp\bigl(\gamma_0 \|X_0\|_{L^2}^2 \bigr)\bigr]\bigr).
\end{align*}
Combining the upper bounds then yields the inequality~\eqref{eq:lemPsi} and concludes the proof of Lemma~\ref{lem:Psi}.
\end{proof}

Finally, in the proof of Theorem~\ref{thm:weakrate} below, the following inequality is employed to obtain convergence rates of error terms with respect to $N$: for all $N,M\in\N$, if $M>N$, one has
\begin{equation}\label{eq:negpowAPNbound}
\|(-A)^{-\alpha}(P_N-P_M)x\|_{L^2} 
\le (\pi N)^{-2\alpha}\|x\|_{L^2},\quad \forall~x\in L^2,
\end{equation}
for all $N,M\in \N$ satisfying $M>N$, see~\eqref{eq:def_fracpowA}.\par 

We are now in a position to prove Theorem~\ref{thm:weakrate}. In the sequel, the objective is to obtain a rate of convergence with respect to $N\in\N$, and an auxiliary integer $M\in\N$ is introduced for technical reasons. The condition $M\ge N$ is imposed below.

\begin{proof}[Proof of Theorem~\ref{thm:weakrate}]
Without loss of generality we can assume that $\alpha \geq \frac12$ and $\delta_0<\frac{1}{4}$. Since $\alpha <1\wedge(\frac34+\delta_0)$, we can choose an auxiliary parameter $\delta\in(0,\delta_0)$ satisfying $\alpha<\frac34+\delta \in (0,1)$. Set $\beta=\frac14-\delta$ and observe that $\beta\in [0,1-\alpha)$.

Note that for all $q\in [1,\infty)$ and 
$\epsilon >0$ there exists $C_{\epsilon,q}\in(0,\infty)$ such that
\begin{equation}\label{eq:momentexpbounds}
\E [ \| X_0 \|_{L^2}^{q} ] \leq C_{\epsilon,q} \E[ \exp(\epsilon \| X_0 \|_{L^2}^{2} )].
\end{equation}
For all $N\in\N$, define
\[
\varepsilon_{N}=\E[\varphi(X_N(T))]-\E[\varphi(X(T))],
\]
and for all $M>N$ set
\[
\varepsilon_{N,M}=\E[\varphi(X_N(T))]-\E[\varphi(X_M(T))].
\]
It follows from the strong convergence result~\eqref{eq:strongGalConvRates} that for all $N\in\N$ one has
\begin{equation}\label{eq:eps_N_lim_eps_NM}
\varepsilon_{N}=\E[\varphi(X_N(T))]-\E[\varphi(X(T))]=\underset{M\to\infty}\lim~\bigl(\E[\varphi(X_N(T))]-\E[\varphi(X_M(T))]\bigr)
=\lim_{M\rightarrow \infty} \varepsilon_{N,M}.
\end{equation}
It thus suffices to prove upper bounds for $|\varepsilon_{N,M}|$ which are independent of the auxiliary parameter $M$ under the condition $M>N$, to obtain upper bounds for $|\varepsilon_{N}|$.

Recall that the mapping $u_M$ is defined by Equation~\ref{eq:Kolmogorov}. Applying the identity~\eqref{eq:Kolmogorov_randomIV} from Lemma~\ref{lemma:Kolmogorov_randomIV} and the tower property of conditional expectation, one obtains
\begin{align}\label{eq:errorsplit1}
\epsilon_{N,M}&=\E[u_M(0,X_N(T))]-\E[u_M(T,X_M(0))]=\eps_{N,M}^{1} +\eps_{N,M}^2,
\end{align}
where for all $M>N$ the auxiliary error terms are defined by
\begin{align*}
    \epsilon_{N,M}^1&=\E[u_M(0,X_N(T))]-\E[u_M(T,X_N(0))],\\
    \epsilon_{N,M}^2&=\E[u_M(T,X_N(0))]-\E[u_M(T,X_M(0))]=\E [u_M(T,P_N X_0)]-\E[u_M(T,P_MX_0)].
\end{align*}

We deal with the error terms $\epsilon_{N,M}^1$ and $\epsilon_{N,M}^2$ separately.

{\bf Treatment of $\epsilon_{N,M}^1$.}
The mapping $u_M$ defined by~\eqref{eq:Kolmogorov} is of class $C^{1,2}([0,T]\times H_M, \R)$. Therefore, applying the It\^o formula to the stochastic process $[0,T]\ni t \mapsto u_M(T-t,X_N(t))$, using the evolution equation~\eqref{eq:GalsolBurgers} for $X_N$ and the definition~\eqref{eq:defWQ} of the Wiener process $W^Q$, one obtains, for all $M>N$:
\begin{align*}
\epsilon_{N,M}^1&=\E\big[u_M(0,X_N(T))]-u_M(T,X_N(0))\big]\\
&=-\int_0^T\E\big[ \tfrac{\partial u_M}{\partial t}(T-t,X_N(t))\big]\,dt\\
&\quad +\int_0^T \E\bigl[Du_M(T-t,X_N(t)).\bigl(AX_N(t)+B_N(X_N(t))\bigr)\bigr]\,dt\\
&\quad +\tfrac12\int_0^T\sum_{j\in \N} q_j \E[D^2u_M(T-t,X_N(t)).\bigl(P_N e_j, P_N e_j\bigr)]\,dt,
\end{align*}\par 
Next, recall that the mapping $u_M$ is solution to the Kolmogorov equation~\eqref{eq:Kolmogorov_eqn}. Therefore, $\epsilon_{N,M}^{1}$ is decomposed as
\begin{equation}\label{eq:errorsplit2}
\epsilon_{N,M}^1=\epsilon_{N,M}^{1,1}+\epsilon_{N,M}^{1,2},
\end{equation}
where for all $M>N$ the error terms $\epsilon_{N,M}^{1,1}$ and $\epsilon_{N,M}^{1,2}$ are defined by
\begin{align*}
    \epsilon_{N,M}^{1,1}&= \int_0^T \E\bigl[Du_M(T-t,X_N(t)).\bigl(B_N(X_N(t))-B_M(X_N(t))\bigr)\bigr]\,dt,\\
    \epsilon_{N,M}^{1,2}&=\frac12\int_0^T\sum_{j\in \N} q_j \E\bigl[D^2u_M(T-t,X_N(t)).\bigl((P_N-P_M) e_j, (P_N+P_M) e_j\bigr)\bigr]\,dt.
\end{align*}

$\bullet$ {Treatment of $\epsilon_{N,M}^{1,1}$.} \par 
Recall that it is assumed that $M>N$, thus $X_N(t)\in H_N\subseteq H_M$, and thus for all $t\in[0,T]$ one has
\[
B_N(X_N(t))-B_M(X_N(t))=(P_N-P_M)B(X_N(t)).
\]

Recall that $\beta=\frac14-\delta\in[0,1-\alpha)$. Applying the inequality~\eqref{eq:negpowAPNbound} and the inequality~\eqref{eq:Kol_deriv_est} from Theorem~\ref{theo:Kol_reg} regarding the first order derivative $Du_M(t,\cdot)$, 
for all $\epsilon,\delta \in (0,\infty)$, there exists $C_{\alpha,\delta,\eps}(T,Q,\varphi)\in (0,\infty)$ such that for all $M>N$ and for all $t\in(0,T)$ one has
\begin{align*}
\big|\E\bigl[&Du_M(T-t,X_N(t)).\bigl(B_N(X_N(t))-B_M(X_N(t))\bigr)\bigr]\big|\\
&\le C_{\alpha, \delta,\eps}(T,Q,\varphi)(T-t)^{-(\alpha+\beta)}\E\bigl[\Psi_{\delta,\epsilon,6}(X_N(t))\|(-A)^{-(\alpha+\beta)}(P_N-P_M)B(X_N(t))\|_{L^2}\bigr]\\
&\le 2 C_{\alpha, \delta, \eps}(T,Q,\varphi) N^{-2(\alpha+\beta)}
(T-t)^{-(\alpha+\beta)}\E\bigl[\Psi_{\delta,\epsilon,6}(X_N(t))
\|B(X_N(t))\|_{L^2}\bigr],
\end{align*}
where $\Psi_{\delta,\epsilon,6}(X_N(t))$ is defined by~\eqref{eq:defPsi}.

Recall from~\eqref{eq:H1_algebra} that the Sobolev space $W^{1,2}$ is an algebra. This observation combined with the Poincar\'e inequality~\eqref{eq:poincare} and with Lemma~\ref{lem:deriv_bdd_fracdomains} yields
\begin{equation}\label{eq:boundBXN}
\|B(X_N(t))\|_{L^2}\le \|X_N(t)^2\|_{W^{1,2}}\le \|X_N(t)\|_{W^{1,2}}^2\le \tfrac{3}{2} \| \nabla X_N(t) \|_{L^2}^2 
= \tfrac{3}{2}\|(-A)^{\frac12}X_N(t)\|_{L^2}^2,
\end{equation}
where the standard definition $\|\cdot\|_{W^{1,2}}^2=\|\cdot\|_{L^2}^{2}+\|\nabla\cdot\|_{L^2}^{2}$ is considered.
Notice that in Section~\ref{sec:bounds} there are no moment bounds for $\|(-A)^{\frac12}X_N(t)\|_{L^2}$, whereas the inequality~\eqref{eq:XM_moment_mixed_reg} from Lemma~\ref{lem:X_sup_moments} provides moment bounds for $\|(-A)^{\frac12-\frac{\delta}{2}}X_N(t)\|_{L^2}$, for positive $\delta$, uniformly with respect to $t\in[0,T]$. Those moment bounds can be exploited by applying the inverse inequality  on $H_N$: for all $\beta\in(0,\frac12)$ and $N\in\N$ one has
\begin{equation}\label{eq:fracpowAPNbound}
\|(-A)^{\frac12}x\|_{L^2}\le (\pi N)^{2\beta}\|(-A)^{\frac12- \beta} x\|_{L^2},\quad \forall~x\in H_N.
\end{equation}

Combining the upper bounds above, recalling that $\frac12-\beta = \frac14+\delta$, one thus obtains the upper bounds
\begin{align*}
\E\bigl[\Psi_{\delta,\epsilon,6}(X_N(t))
\|B(X_N(t))\|_{L^2}\bigr]&\le \frac{3(\pi N)^{2\beta}}{2}\E\bigl[\Psi_{\delta,\epsilon,6}(X_N(t))\|(-A)^{\frac12-\beta}X_N(t)\|_{L^2}^2\bigr]\\
&\le \frac{3(\pi N)^{2\beta}}{2}\E\bigl[\Psi_{\delta,\epsilon,8}(X_N(t))\bigr].
\end{align*}
Therefore for all $M>N$ there exists a $C_{\alpha,\delta,\epsilon}(T,Q,\varphi)\in (0,\infty)$ such that for all $M>N$ one has
\begin{equation}\label{eq:eps11bound1}
|\epsilon_{N,M}^{1,1}|\le C_{\alpha,\delta,\epsilon}(T,Q,\varphi)
N^{-2\alpha}\int_0^T (T-t)^{-(\alpha+\beta)}\,dt \sup_{t\in [0,T]} \E\bigl[\Psi_{\delta,\epsilon,8}(X_N(t))\bigr].
\end{equation}
Let $\epsilon=\frac{p-2q}{2p}\frac{\gamma_0}{1+2\gamma_0\|Q\|_{\mathcal{L}(L^2)}}$ with $q=8$. Applying the inequality~\eqref{eq:lemPsi} from Lemma~\ref{lem:Psi}, one obtains the following inequality: for all $M>N$ there exists a $C_{\alpha,\gamma_0,\delta,\delta_0,p}(T,Q,\varphi)\in (0,\infty)$ such that
\begin{equation}\label{eq:boundeps11}
\begin{aligned}
|\epsilon_{N,M}^{1,1}| &\le C_{\alpha,\gamma_0,\delta,\delta_0,p}(T,Q,\varphi)N^{-2\alpha} 
\left(1+\E\bigl[ \exp(\gamma_0 \| X_0 \|_{L^2}^2 ) \bigr] \right)^2
\Bigl( 1+\bigl( \E \bigl[\| (-A)^{\frac{1}{4}+\delta_0} X_0 \|_{L^2}^{p}\bigr]\bigr)\Bigr)^{\frac{16}{p}}.
\end{aligned}
\end{equation}

$\bullet$ {Treatment of $\epsilon_{N,M}^{1,2}$.}

Owing to the inequality~\eqref{eq:Kol_2nd_deriv_est} from Theorem~\ref{theo:Kol_reg} regarding the second order derivative $D^2u_M(t,\cdot)$ (applied with $\beta=0$), 
for all $\eps,\delta \in (0,\infty)$ there exists a $C_{\alpha,\delta,\epsilon}(T,Q,\varphi)\in (0,\infty)$ such that for all $t\in(0,T)$ and all $M>N$ one has
\begin{align*}
\big|\E\bigl[&D^2u_M(T-t,X_N(t)).\bigl( (P_N-P_M)e_j,(P_N+P_M)e_j\bigr)]\big|\\
&\le C_{\alpha,\delta,\epsilon}(T,Q,\varphi)(T-t)^{-\alpha}\E[\Psi_{\delta,\epsilon,16}(X_N(t))]\|(-A)^{-\alpha}(P_N-P_M)e_j\|_{L^2}\|(P_N+P_M)e_j\|_{L^2}\\
&\le 2 C_{\alpha,\delta,\epsilon}(T,Q,\varphi)(T-t)^{-\alpha}\E[\Psi_{\delta,\epsilon,16}(X_N(t))]N^{-2\alpha},
\end{align*}
using the inequality~\eqref{eq:negpowAPNbound} in the last step, and where $\Psi_{\delta,\epsilon,16}(X_N(t))$ is defined by~\eqref{eq:defPsi}.

Recall that $\sum_{j\in_N}q_j={\rm Tr}(Q)<\infty$. There thus exists $C_{\alpha,\delta,\eps}(T,Q,\varphi)\in (0,\infty)$ such that for all $M>N$ one has
\begin{equation}\label{eq:eps12bound1}
|\epsilon_{N,M}^{1,2}|\le C_{\alpha,\delta,\eps}(T,Q,\varphi)
N^{-2\alpha}\int_0^T (T-t)^{-(\alpha+\beta)}\,dt \sup_{t\in [0,T]}  \E\bigl[\Psi_{\delta,\epsilon,16}(X_N(t))\bigr].
\end{equation}
Let $\epsilon=\frac{p-2q}{2p}\frac{\gamma_0}{1+2\gamma_0\|Q\|_{\mathcal{L}(L^2)}}$ with $q=16$. Applying the inequality~\eqref{eq:lemPsi} from Lemma~\ref{lem:Psi}, one obtains that there exists a $C_{\alpha,\gamma_0,\delta,\delta_0,p}(T,Q,\varphi)\in (0,\infty)$ such that for all $M>N$ one has
\begin{equation}\label{eq:boundeps12}
\begin{aligned}
|\epsilon_{N,M}^{1,2}| &\le C_{\alpha,\gamma_0,\delta,\delta_0,p}(T,Q,\varphi)N^{-2\alpha} 
\left(1+\E\bigl[ \exp(\gamma_0 \| X_0 \|_{L^2}^2 ) \bigr] \right)^2
\Bigl( 1+\bigl( \E \bigl[\| (-A)^{\frac{1}{4}+\delta_0} X_0 \|_{L^2}^{p}\bigr]\bigr)\Bigr)^{\frac{32}{p}}.
\end{aligned}
\end{equation}

{\bf Treatment of $\epsilon_{N,M}^2$.}\par 

For all $M>N$, one has
\begin{align*}
|\epsilon_{N,M}^2|&=|\E [u_M(T,P_N X_0)-u_M(T,P_M X_0)]|\\
& =  \Big|\E \Big[
    \int_{0}^{1} Du_M(T,(P_M X_0+r(P_N- P_M) X_0)).(P_N X_0 - P_M X_0) \,dr \Big]\Big|.
\end{align*}
Since $P_N$ and $P_M$ are the orthogonal projections on the eigenspaces $H_N$ and $H_M$ of $A$, one has 
\begin{equation}\label{eq:simpleAPNest}
\| (-A)^{\frac14+\delta_0} (P_M X_0 + r(P_N-P_M)X_0 ) \|_{L^2}
\leq 
\| (-A)^{\frac14+\delta_0} X_0 \|_{L^2},\quad \forall~r\in[0,1].
\end{equation}
Owing to this and to the inequality~\eqref{eq:Kol_deriv_est} from Theorem~\ref{theo:Kol_reg}  (with $\epsilon=\gamma_0/3$) regarding the first order derivative $Du_M(t,\cdot)$, there exists $C_{\alpha,\gamma_0,\delta_0}(T,Q,\varphi)\in (0,\infty)$ such that for all $M>N$ one has
\begin{equation*}
\begin{aligned}
|\epsilon_{N,M}^2|
&\le C_{\alpha,\gamma_0,\delta_0}(T,Q,\varphi)
\E\big[
    \exp\big(\tfrac{\gamma_0}{3} \|X_0\|_{L^2}^2 \big) 
    \big( 1+\|(-A)^{\frac{1}{4}+\delta_0 } X_0\|_{L^2}^{6}\big)
    \|(-A)^{-\alpha}(P_N-P_M)X_0\|_{L^2}
    \big]\\
&\le C_{\alpha,\gamma_0,\delta_0}(T,Q,\varphi)
    N^{-2\alpha}\E[
    \exp\big(\tfrac{\gamma_0}{3} \|X_0\|_{L^2}^2 \big) 
    \big( 1+\|(-A)^{\frac{1}{4}+\delta_0 } X_0\|_{L^2}^{6}\big)
    \|X_0\|_{L^2}
    \big].
\end{aligned}
\end{equation*}
Therefore applying~\eqref{eq:momentexpbounds} and H\"older's inequality, one obtains that there exists a $C_{\alpha,\gamma_0,\delta_0}(T,Q,\varphi)\in (0,1)$ such that for all $M>N$ one has
\begin{equation}
\begin{aligned}\label{eq:boundeps2}
|\epsilon_{N,M}^2|
&
\le C_{\alpha,\gamma_0,\delta_0}(T,Q,\varphi) 
N^{-2\alpha} 
\bigl( \E[
    \exp\big(\gamma_0 \|X_0\|_{L^2}^2 \big) ]\bigr)^{\frac{2}{3}}
    \bigl( 1+\bigl( \E[\|(-A)^{\frac{1}{4}+\delta_0 } X_0\|_{L^2}^{18}\bigr)^{\frac{1}{3}} \bigr).
\end{aligned}
\end{equation}
\par \medskip 

{\bf Conclusion.} \par
Combining~\eqref{eq:eps_N_lim_eps_NM},~\eqref{eq:errorsplit1},~\eqref{eq:errorsplit2},~\eqref{eq:boundeps11},~\eqref{eq:boundeps12}
and~\eqref{eq:boundeps2} completes the proof of Thoerem~\ref{thm:weakrate}.
\end{proof}

\appendix
\section{Proof of the moment bounds}\label{app}

\subsection{Proof of Lemma~\ref{lem:expest}} \label{subsectionA1}

\begin{proof}[Proof of Lemma~\ref{lem:expest}]

\eqref{it:momentbounds}
Using the property~\eqref{eq:BsymM2}, applying It\^o's formula for $\|.\|_{L^2}^2$ and performing an integration by parts, we deduce the following energy equality, for any stopping time $\tau \colon \Omega \rightarrow [0,T]$:
\begin{equation}		\label{energy-Id}
\|X_M(\tau)\|_{L^2}^2 + 2 \int_0^\tau \|\nabla X_M(s)\|_{L^2}^2 \,ds = \|P_M X_0\|_{L^2}^2 + 2 \int_0^\tau \langle X_M(s),dW^Q(s) \rangle_{L^2}  + \tau \Tr(Q).
\end{equation}
For any 
$R\in (0,\infty)$  define the stopping time $\tau_R:=\inf\{ t\in[0,T]\, : \, \|X_M(t)\|_{L^2}\geq R\}$. Owing to global well-posedness of~\eqref{eq:GalsolBurgers}, note that $\lim_{R\rightarrow \infty} \tau_R = T$.

Let $p\in [4, \infty)$. Applying the It\^o formula to~\eqref{energy-Id} for the $C^2$-function $\xi\mapsto \xi^{\frac{p}{2}}$, one has for any  $t\in [0,T]$
\begin{align}\label{eq:pthmomentIto}
\|&X_M(t \wedge \tau_R)\|_{L^2}^{p} + p \int_0^{t \wedge \tau_R} \!\! \|\nabla X_M(s)\|_{L^2}^2 \|X_M(s)\|_{L^2}^{p-2} \,ds \nonumber 
\\ & = \|P_M X_0\|_{L^2}^{p}  
+ p \int_0^{t \wedge \tau_R} \!\! 
  \|X_M(s)\|_{L^2}^{p-2} \langle X_M(s),dW^Q(s) \rangle_{L^2}   \nonumber 
\\ 
&  \quad 
+  \tfrac{p}{2} \Tr(Q)\! \int_0^{\tau\wedge \tau_R} \!\! \|X_M(s)\|_{L^2}^{p-2}\,ds
+ \tfrac{p}{4}(p-2) \int_0^{\tau\wedge \tau_R} \!\!  \|X_M(s)\|_{L^2}^{p-4} \|\sqrt{Q} X_M(s)\|_{L^2}^2 \,ds.
\end{align}
The linear operator $\sqrt{Q}$ is bounded with $\|\sqrt{Q}\|_{\mathcal{L}(L^2)}^{2}=\|Q\|_{\mathcal{L}(L^2)}\le {\rm Tr}(Q)$;  thus taking expectations (note that $\| X_M(t\wedge \tau_R)\|_{L^2}\leq R + \| X_0 \|_{L^2}$ a.s. for all $t\in[0,T]$, therefore the stochastic integral above is a martingale) we deduce that for any $t\in [0,T]$ 
\begin{align*}
&\E\bigl[ \|X_M(t \wedge \tau_R)\|_{L^2}^{p} \bigr] + p \E\Bigl[  \int_0^{t \wedge \tau_R} \!\! \|\nabla X_M(s)\|_{L^2}^2 \|X_M(s)\|_{L^2}^{p-2} \,ds  \Bigr] 
\\ & \qquad \leq \E[\|P_MX_0\|_{L^2}^{p}] 
+ \tfrac{p^2}{4}  \Tr(Q)\;  \E\Bigl[ \int_0^t \bigl( \|X_M(s\wedge \tau_R)\|_{L^2}^{p} +1\bigr) \,ds\Bigr].
\end{align*} 
Recalling that $\|P_MX_0\|_{L^2}\le \|X_0\|_{L^2}$, neglecting the second term in the above left hand side and applying the Gr\"onwall lemma, 
we deduce that there exists $C_{p}(T,Q) \in (0,\infty)$ such that
\begin{equation}\label{sup_E_2p}
\underset{M\in \N }\sup~
\underset{R\in (0,\infty)}\sup~\underset{t\in [0,T]}\sup~\E\bigl[ \|X_M(t\wedge \tau_R)\|_{L^2}^{p} \bigr] \leq  C_{p}(T,Q)  \E[ \| X_0 \|_{L^2}^{p}].
\end{equation}
Furthermore, the Burkholder--Davis--Gundy inequality (see, e.g.,~\cite[Theorem 3.14 and Theorem 4.12]{DaPratoZabczyk:1992})  yields 
\begin{align*}
    &\E\Bigl[\underset{t\in [0,T]}\sup~\Big|\int_0^{t\wedge \tau_R}   \|X_M(s)\|_{L^2}^{p-2} \langle X_M(s),dW^Q(s) \rangle_{L^2}\Big|\Bigr]\\
    & \qquad \le 3\E\Bigl[\Bigl(\int_{0}^{T\wedge \tau_R}\|X_M(s)\|_{L^2}^{2p-4}\|\sqrt{Q}X_M(s)\|_{L^2}^{2}  \,ds \Bigr)^{\frac12}\Bigr]\\
    &\qquad \le 3{\rm Tr}(Q)\E\Bigl[\underset{s\in[0,T]}\sup~\|X_M(s\wedge\tau_R)\|_{L^2}^{\frac{p}{2}} \Bigl(\int_{0}^{T\wedge \tau_R}\|X_M(s)\|_{L^2}^{p-2} \,ds \Bigr)^{\frac12}\Bigr].
\end{align*}
Therefore, owing to~\eqref{eq:pthmomentIto} we deduce  that for every $t\in [0,T]$
\begin{align*}
&\E\Bigl[ \underset{t\in [0,T]}\sup~\Bigl( \|X_M(t\wedge \tau_R)\|_{L^2}^{p} + p \int_0^{t\wedge \tau_R} \!\! \|\nabla X_M(s)\|_{L^2}^2 \|X_M(s)\|_{L^2}^{p-2} \,ds  \Bigr)\Bigr]  
\\ & \qquad 
\leq \E[\|X_0\|_{L^2}^{p}] 
+ 3p \Tr(Q)\; \E\Bigl[\underset{s\in [0,T]}\sup~\|X_M(s\wedge \tau_R)\|_{L^2}^{\frac{p}{2}} \bigl( \; \int_0^{T\wedge \tau_R} \|X_M(s)\|_{L^2}^{p-2}\,ds \bigr)^{\frac{1}{2}} \Bigr]\\
& \qquad \quad +  \frac{p^2}{4} \Tr(Q)\;  \E\Bigl[ \int_0^{T}\bigl( \|X_M(s\wedge \tau_R)\|_{L^2}^{p} +1\bigr) \,ds\Bigr]\\
& \qquad \leq \E[\|X_0\|_{L^2}^{p}]  +  \tfrac{1}{2} \E\bigl[ \underset{s\in [0,T]}\sup~\|X_M(s\wedge \tau_R)\|_{L^2}^{p} \bigr]   
\\ & \qquad \quad 
+ \frac{19 p^2}{4} \Tr(Q)  \; \E\bigl[ \int_0^T \bigl( \|X_M(s\wedge \tau_R)\|_{L^2}^{p} +1\bigr) \,ds\bigr].
\end{align*}

Note that $\underset{t\in [0,T]}\sup~\|X_M(t\wedge \tau_R)\|_{L^2}^{p}\leq R^{p}+\|X_0\|_{L^2}^{p}$, thus $\E\bigl[ \underset{s\in [0,T]}\sup~\|X_M(s\wedge \tau_R)\|_{L^2}^{p} \bigr]<\infty$. As a result, applying the upper bound~\eqref{sup_E_2p} shows that there exists $C_p(T,Q)\in (0,\infty)$ such that
\begin{align*}
\underset{M\in \N}\sup~\underset{R\in (0,\infty)}\sup~ &\E\Bigl[ \underset{t\in [0,T]}\sup~
\Bigl( \|X_M(t\wedge \tau_R)\|_{L^2}^{p} + p \int_0^{t\wedge \tau_R} \!\! \|\nabla X_M(s)\|_{L^2}^2 \|X_M(s)\|_{L^2}^{p-2} \,ds  \Bigr)\Bigr]  
\\ & \qquad \leq C_{p}(T,Q) \E[ \| X_0 \|_{L^2}^{p}].
\end{align*}

Letting  $R\to \infty$ and applying the monotone convergence theorem implies that the upper bound~\eqref{pmoments} holds for every $p\in [4,\infty)$. \par 
\eqref{it:expmomentbounds}
Multiplying the identity~\eqref{energy-Id} (with $\tau=t$) by 
$\beta\in (0,\infty)$, we deduce 
\begin{align} \beta \|X_M(t)\|_{L^2}^2 +& \beta \int_0^t \| \nabla X_M(s) \|_{L^2}^2 \,ds 
= \beta \|P_M X_0\|_{L^2}^2 +\beta t \Tr(Q)  \nonumber
\\ & \quad +2\beta \int_0^t \langle X_M(s)\,dW^Q(s) \rangle_{L^2}- \beta \int_0^t \| \nabla X_M(s)\|_{L^2}^2 \,ds , \quad \forall t\geq 0. \label{eq:aux_expo_mom}
\end{align}
For all $\beta \in (0,\infty)$ and all $t\in [0,T]$ set
\[
N^{(\beta)}(t)=2\beta \int_0^t \langle X_M(s),dW^Q(s) \rangle_{L^2},
\]
and define the random variable
\begin{equation}\label{eq:Ybeta}
\begin{aligned}
Y^{(\beta)}&= \exp\Bigl( \beta \underset{t\in [0,T]}\sup~\bigl( 2\int_0^t \langle X_M(s),
dW^Q(s)\rangle_{L^2} -\int_0^t \|\nabla X_M(s)\|_{L^2}^2 \,ds \bigr) \Bigr)
\\ & 
=\exp\Bigl( \underset{t\in [0,T]}\sup~\bigl(N^{(\beta)}(t)-\beta\int_0^t \|\nabla X_M(s)\|_{L^2}^2 \,ds \bigr) \Bigr).
\end{aligned}
\end{equation}
The stochastic process $(N^{(\beta)}(t))_{0\leq t\leq T}$ is a square integrable martingale, and its quadratic variation $\bigl(\langle N^{(\beta)}\rangle_t\bigr)_{t\ge 0}$ is given by
\[
\langle N^{(\beta)}\rangle_t = 4 \beta^2  
\int_0^t \|\sqrt{Q} X_M(s)\|_{L^2}^2 \,ds,\quad\forall~t\ge 0.
\]
Owing to the Poincar\'e inequality~\eqref{eq:poincare} and observing that $\|\sqrt{Q}\|_{\mathcal{L}(L^2)}^2=\|Q\|_{\mathcal{L}(L^2)}$, 
 \[
\langle N^{(\beta)}\rangle_t \leq  2 \beta^2 \| Q \|_{\mathcal{L}(L^2)} 
\int_0^t \| \nabla X_M(s)\|_{L^2}^2 \,ds, \quad \forall t\geq 0.
\]
To ease notations, set  $c(\beta) = \tfrac{1}{2 \beta \| Q \|_{\mathcal{L}(L^2)}}$. Combining equation~\eqref{eq:Ybeta} with the above upper bound one obtains for all $t\in[0,T]$
\[
-\beta \int_0^t \| \nabla X_M(s)\|_{L^2}^2 \,ds \le -c(\beta)\langle N^{(\beta)}\rangle_t,
\]
and thus the upper bound
\[
Y^{(\beta)}\le \exp\Bigl( \underset{t\in [0,T]}\sup~\bigl(N^{(\beta)}(t) - c(\beta) \langle N^{(\beta)}\rangle_t\bigr)\Bigr).
\]
For all $t\in [0,T]$, set
\[
\hat{N}^{(\beta)}(t)=2c(\beta)N^{(\beta)}(t),\quad M^{(\beta)}(t)=\exp\bigl(\hat{N}^{(\beta)}(t)-\frac12\langle \hat{N}^{(\beta)}\rangle_t\bigr).
\]
For all $K\in(0,\infty)$, owing to the Markov inequality we deduce 
\begin{align*}
    \PP(Y^{(\beta)}\ge e^K)&=\PP\bigl((Y^{(\beta)})^{2c(\beta)}\ge e^{2c(\beta)K})\\
    &\le\PP\left(\exp\Big(2c(\beta)\underset{t\in [0,T]}\sup~\bigl(N^{(\beta)}(t) - c(\beta) \langle N^{(\beta)}\rangle_t\bigr)\Bigr) \ge e^{2c(\beta)K} \right)\\
    &\le \PP\left(\exp\Big(\underset{t\in [0,T]}\sup~\bigl(\hat{N}^{(\beta)}(t) - \frac12\langle \hat{N}^{(\beta)}\rangle_t\bigr)\Bigr) \ge e^{2c(\beta)K} \right)\\
    &\le \PP\Bigl(\underset{t\in [0,T]}\sup~M^{(\beta)}(t)\ge e^{2c(\beta)K} \Bigr).
\end{align*}

Note that (see e.g.  \cite[Proposition~3.4, p.~140]{RevuzYor:1991})  the process $M^{(\beta)}$ is a (continuous) local martingale, and hence  (see e.g. \cite[Exercise~1.46 p.~129]{RevuzYor:1991})
that  $M^{(\beta)}$ is a supermartingale. Finally, since $M^{(\beta)}(0)=1$, \cite[Exercise~1.15, p. 55]{RevuzYor:1991} implies the following maximal inequality
\[
\PP\Big( \sup_{t\in [0;T]} M^{(\beta)}(t) \geq  e^{2 c(\beta) K}\Big) 
\leq e^{-2c(\beta)K}.
\]

Assume from now on that $\beta\in(0,\frac{1}{2\| Q \|_{\mathcal{L}(L^2)}})$; this implies that $c(\beta)>1$ and that
\[
\PP(Y^{(\beta)}\ge e^K)\le e^{-2K}, \quad \forall K>0.
\]
Therefore one obtains
\begin{align*}
\E(Y^{(\beta)})&=\int_{0}^{\infty}\PP(Y^{(\beta)}>y)\,dy 
\le 1+\int_{1}^{\infty}\PP(Y^{(\beta)}\ge y)\,dy 
\le 1+\int_0^\infty \PP\big(Y^{(\beta)}\ge e^K\big) e^K \,dK \nonumber \\
&
\le 1+\int_0^\infty e^{-K} \,dK=2, \quad \forall \beta\in\left(0,\frac{1}{2\| Q \|_{\mathcal{L}(L^2)}}\right).
\end{align*}
Using the identity~\eqref{eq:aux_expo_mom} and  the definition~\eqref{eq:Ybeta} of the random variable $Y^{(\beta)}$, we deduce
\[
\E\Bigl[ \exp \Bigl( \beta \underset{t\in [0,T]}\sup~ \|X_M(t)\|_{L^2}^2
+ \beta \int_0^T \!\! \| \nabla X_M(s)\|_{L^2}^2 \,ds \Bigr) \Bigr]
\le  \E \bigl[Y^{(\beta)} \exp\bigl( \beta \| X_0\|_{L^2}^2\bigr) \bigr] \exp(\beta T{\rm Tr}(Q)).
\]
Recall that the condition $\beta < \frac{\gamma_0}{1+ 2  \gamma_0 \| Q \|_{\mathcal{L}(L^2)}}$ is assumed to be satisfied. 
Applying H\"older's inequality with the conjugate exponents $p=\frac{\gamma_0}{\beta}$ and $q=\frac{\gamma_0}{\gamma_0-\beta}$, we deduce 
\[
\E \bigl[Y^{(\beta)} \exp\bigl( \beta \| X_0\|_{L^2}^2\bigr) \bigr]\le \bigl(\E\bigl[(Y^{(\beta)})^{\frac{\gamma_0}{\gamma_0-\beta}}\bigr]\bigr)^{\frac{\gamma_0-\beta}{\gamma_0}}\bigl(\E\bigl[\exp\bigl( \gamma_0 \| X_0\|_{L^2}^2\bigr)\bigr]\bigr)^{\frac{\beta}{\gamma_0}}.
\]
Note that  $\E[(Y^{(\beta)})^{\frac{\gamma_0}{\gamma_0-\beta}}]=\E[(Y^{(\frac{\beta \gamma_0}{\gamma_0-\beta})}]$; the condition on $\beta$ above yields the inequality  $\frac{\beta \gamma_0}{\gamma_0-\beta}< \frac{1}{2\| Q \|_{\mathcal{L}(L^2)}}$, thus applying the inequality above one has $\E[(Y^{(q\beta)})] \le 2$.
The proof of the exponential moment bounds~\eqref{expmoments_ranIV} is thus completed.

This concludes the proof of Lemma~\ref{lem:expest}.

\end{proof}

\subsection{Some auxiliary tools}

Before proceeding with the proofs of Lemmas~\ref{lem:momLinfty} and \ref{lem:X_sup_moments}, let us introduce some notation and state auxiliary results.

We first establish regularity results of the stochastic convolution using the factorization method of Da Prato, Kwapie\'n, and Zabczyk~\cite{DaPratoKwapienZabczyk:1987}. 
Such results are classical, see also~\cite[Theorem 10.5]{NeervenVeraarWeis:2008}, so we provide the argument here only for the readers' convenience.

\begin{lemma}\label{lem:reg_stochconv}
Given   $M\in \N$, let $I_M\colon [0,T]\times \Omega \rightarrow L^2$  be defined  by
\begin{equation}\label{eq:stochconv}
 I_M(t) = \int_0^t e^{(t-s)A}P_M\,dW^{Q}(s), \quad \forall  t\in [0,T].
\end{equation}
For all 
 $\lambda,\mu \in (0,1)$ satisfying $\lambda+\mu<\frac{1}{2}$ and all $p\in [1, \infty)$ there exists $C_{\lambda,\mu,p}(T)\in (0,\infty)$  such that $I_M(t)\in D((-A)^{\lambda})$ a.s.\ for all $t\in [0,T]$ and
\begin{align}\label{eq:reg_stochconv_mixed_unbalanced}
\underset{M\in \N}\sup~\E\bigl[ \| (-A)^{\lambda} I_M \|^p_{C^{\mu}([0,T],  L^2)}   \bigr]
&\leq 
C_{\lambda,\mu,p}(T) (\Tr(Q))^{p/2}. 
\end{align}
\end{lemma} 
\begin{proof}
By H\"older's inequality it suffices to prove that~\eqref{eq:reg_stochconv_mixed_unbalanced} holds for $p$ sufficiently large, which can be obtained from the Da Prato--Kwapie\`n--Zabczyk 
factorization method~\cite{DaPratoKwapienZabczyk:1987}. 
More specifically, assume that $p> \frac{2}{1-2(\lambda+\mu)}$ and apply~\cite[Proposition 4.2]{NeervenVeraarWeis:2008} with $\alpha \in (\frac{1}{p}, \frac{1-2(\lambda +\mu)}{2})$ (e.g., take $\alpha$ to be the midpoint of this interval), $\lambda =\mu$, $\eta = \lambda$, 
$\theta = 0$, $E=H=L^2$, $p > \frac{2}{1-2(\lambda+\gamma)}$, and $\Phi \equiv P_M \sqrt{Q}$. Note that in reference  \cite{NeervenVeraarWeis:2008} the stochastic integral is defined in terms of 
an $H$-cylindrical Brownian motion, that $E_0=E=L^2$,  that $E_\eta = D((-A)^\eta)$, 
that since $E_0$ is a Hilbert space (and hence a UMD space, see e.g.~\cite[Proposition 4.2.14]{HytonenNeervenVeraarWeis:2016}), the set $\gamma(L^2(0,t;H), E_{0})$ of radonifying operators from $L^2(0,t;H)$ to $E_0$
coincides with $L^2(0,t ; {\mathcal L}_2(L^2))$  
by~\cite[pages~258-259]{HytonenNeervenVeraarWeis:2017}, and finally that $\| \Phi \|_{\mathcal{L}_2(L^2)} \leq (\Tr(Q))^{1/2}$.
\end{proof}

Let us also introduce the Green function $G:(0,\infty)\times[0,1]^2\to\R$ associated with the operator $\partial_t - A$, which is given by
\[
G(t;y,z) =\sum_{n\in \N} e^{-n^2\pi^2 t} \sin(n\pi y) \sin(n\pi z),\quad \forall~t\in(0,\infty),~\forall~y,z\in[0,1].
\]
Recall the following heat kernel estimates from e.g.~\cite[p. 93]{Gyongy-Rovira},  \cite[Theorem VI.23, p.221]{Eidelman-Zhitarashu} and the references therein: for any integer $k\in \{0,1,2\}$,
 there exist 
 $C_k,b_k\in (0,\infty)$ such that one has
\begin{align}
\left|\tfrac{\partial^k G}{\partial z^k }(t; y,z)\right| \leq& \;  C_k  t^{-\frac{k+1}{2}} \;  \exp\Big( - b_k  \frac{|y-z|^2}{t}\Big),	 \quad \forall t\in (0,\infty), \; \forall y,z\in [0,1]. \label{upper_H}
\end{align}

The Green function is employed to obtain from the mild formulation~\eqref{eq:mildGalsolBurgers} the following expression: 
for all $t\ge 0$ and all $z\in[0,1]$
\begin{align} 	\label{eq:mildGalsolBurgers_conv}
X_M(t,z) =& e^{tA}  P_M X_0(z)  + \int_0^t \int_0^1 G(t-s;z,y) \bigl(P_M \nabla (X_M^2(s,\cdot)\bigr)(y) \,dy \,ds+ I_M(t)(z).
\end{align}
\subsection{Proof of Lemma~\ref{lem:momLinfty}}

\begin{proof}
For all $t\ge 0$ and all $z\in[0,1]$ let 
\begin{align}\label{eq:JXM}
 J(X_M)(t,z)= \int_0^t \int_0^1 G(t-s;z,y) (P_M \nabla  X_M^2(s))(y) \,dy \,ds.
\end{align}

Owing to the mild formulation~\eqref{eq:mildGalsolBurgers_conv} and to the definition~\eqref{eq:stochconv} of the stochastic convolution $I_M$, we  obtain
\begin{align}\label{eq:sup_norm_error_split}
 \underset{(t,z) \in [0,T]\times [0,1]}\sup~ |X_M(t)(z)| 
  \leq &\underset{t\in [0,T]}\sup~\| e^{t A}P_M X_0 \|_{L^{\infty}}
 + \underset{t \in [0,T]}\sup~ \| I_M(t)\|_{L^{\infty}} \nonumber \\
&\; + \underset{(t,z)\in [0,T]\times [0,1]}\sup~ |J(X_M)(t,z)|.
\end{align}
Applying the Sobolev inequality~\eqref{eq:Linfty_DA_bound} and using the condition $\alpha>1/4$, there exists $C_\alpha\in(0,\infty)$ such that 
\begin{equation}\label{eq:IVLinftybound}
\underset{t\in [0,T]}\sup~\| e^{t A}P_M X_0 \|_{L^{\infty}}\le C_\alpha\underset{t\in [0,T]}\sup~\|(-A)^{\alpha}e^{t A}P_M X_0 \|_{L^{2}}\le C_\alpha\|(-A)^{\alpha}X_0 \|_{L^{2}}.
\end{equation}
Moreover, owing again to the the Sobolev inequality~\eqref{eq:Linfty_DA_bound}, if $\lambda\in(\frac{1}{4},\infty)$ and $\mu \in (0,\infty)$, then there exists $C_{\lambda,\mu}\in(0,\infty)$ such that one has
\begin{equation}        \label{A.12Bis}
\underset{t \in [0,T]}\sup~ \| I_M(t)\|_{L^{\infty}}\le C_{\lambda,\mu}\|(-A)^{\lambda} I_M\|_{C^{\mu}([0,T],L^2)}.
\end{equation}
Thus, in view of Lemma~\ref{lem:reg_stochconv}, all that remains is to prove upper bounds for $J(X_M)(t,z)$ defined by~\eqref{eq:JXM}. First, since $P_M$ is a self-adjoint operator, one obtains 
\begin{align*}
    J(X_M)(t,z)
    &=\int_{0}^{t}\langle G(t-s;z,\cdot),P_M\nabla X_M^2(s)\rangle_{L^2}\,ds 
    =-\int_{0}^{t}\langle \nabla P_MG(t-s;z,\cdot),X_M(s,\cdot)^2\rangle_{L^2}  \,ds.
\end{align*}
Let $\tilde{p}=\frac{3p}{4}\in [2,\infty)$
and $\tilde{q}=\frac{\tilde{p}}{\tilde{p}-1}$ be conjugate exponents, where $p\in[\frac83,\infty)$ is given in the statement of Lemma~\ref{lem:momLinfty}. Since $L^\infty \subseteq L^{\tilde{q}}$, one has
\[
\|X_M^2(s,.)\|_{L^{\tilde{q}}}\le \|X_M(s,.)\|_{L^{\infty}}^{\frac12}\|X_M^2(s,.)\|_{L^{\tilde{q}}}^{\frac34},
\]
and applying H\"older's inequality 
we deduce
\begin{align}
    |J(X_M)(t,z)|&\le \ \int_0^t \| \nabla P_M G(t-s; z,\cdot )\|_{L^{\tilde{p}}} \|X_M^2(s,.)\|_{L^{\tilde{q}}}  \,ds \nonumber\\
    &\le \ \underset{(s,y)\in[0,t]\times[0,1]}\sup~|X_M(s,y)|^{\frac12}\int_0^t \| \nabla P_M G(t-s; z,\cdot )\|_{L^{\tilde{p}}} \|X_M^2(s,.)\|_{L^{\tilde{q}}}^{\frac34}  \,ds \label{Holder-J(X_M)}.
\end{align}
On the
 one hand, 
 the Gagliardo--Nirenberg inequality~\eqref{eq:GN} yields
\begin{align*} 
\| \nabla P_M G(t-s; z, \cdot)\|_{L^{\tilde{p}}} &\leq C  \| \nabla P_M G(t-s;z, \cdot)\|_{L^2}^{1/2+1/{\tilde{p}}}~\| \nabla^2 P_M G(t-s;z, \cdot)\|_{L^2}^{1/2-1/{\tilde{p}}} \\
& \quad + C \| \nabla P_M G(t-s;z, \cdot)\|_{L^2}.
\end{align*}
Observe that $G(t,z,\cdot) \in W^{1,2}_0\cap W^{2,2}$ for all $t\in (0,\infty)$ and all $z\in [0,1]$, and note that $\nabla^2 P_M = P_M \nabla^{2}$ on $W^{1,2}_0\cap W^{2,2}$ owing to~\eqref{eq:PMnabla} and~\eqref{eq:nablaPM}.  Using the identity~\eqref{eq:nablaPM}, the fact that $P_M$ and $Q_M$ are $L^2$-orthogonal projections, 
and the heat kernel estimates~\eqref{upper_H} with $k=1,2$, as a result there exists $C\in (0,\infty)$ such that for all $t>s\ge 0$ and $z\in[0,1]$ 
\begin{align} 	\label{upp_GN_1}
& \underset{M\in\N}\sup~ \| P_M \nabla G(t-s; z, \cdot)\|_{L^{\tilde{p}}} \nonumber \\
&\qquad \leq  C\left( \|\nabla G(t-s;z, \cdot)\|_{L^2}^{1/2+1/{\tilde{p}}} ~ \| \nabla^2 G(t-s; z, \cdot)\|_{L^2}^{1/2-1/{\tilde{p}}} 
+ \| \nabla G(t-s;z, \cdot)\|_{L^2}\right)   \nonumber \\
& \qquad \leq C \; (t-s)^{\frac{1}{2{\tilde{p}}}-\frac{5}{4}}~\Big\| \exp\Big( -b_1 \frac{|z-\cdot |^2}{t-s} \Big) \Big\|_{L^2}^{1/2 +1/{\tilde{p}}} \quad \Big\| \exp\Big( -b_2 \frac{|z-\cdot |^2}{t-s} \Big) \Big\|_{L^2}^{1/2-1/{\tilde{p}}}    \nonumber \\
& \qquad \quad + C\; (t-s)^{-1}~\Big\| \exp\Big( -b_1 \frac{|z-\cdot |^2}{t-s} \Big) \Big\|_{L^2} \nonumber\\
& \qquad \leq  C \bigl( (t-s)^{\frac{1}{2\tilde{p}}-1} + (t-s)^{-\frac{3}{4}}\bigr) \leq  C(1+T^{\frac{1}{4}-\frac{1}{2\tilde{p}}}) (t-s)^{\frac{1}{2\tilde{p}}-1} . 
\end{align}

On the other hand, applying again the Gagliardo--Nirenberg inequality~\eqref{eq:GN} and the Poincar\'e inequality (noting that $X_M$ takes values in $H_M\subseteq W^{1,2}_0$) and using the condition $\frac{1}{{\tilde{p}}}+\frac{1}{{\tilde{q}}}=1$,  we deduce the existence of 
$C\in(0,\infty)$ such that 
\begin{align}		\label{upp_GN_2}
 \|X_M^2(s,.)\|_{L^{\tilde{q}}} = \|X_M(s,.)\|^2_{L^{2{\tilde{q}}}(0,1)} \leq & \;  C \|X_M(s,.) \|_{L^2}^{2-1/{\tilde{p}}}~\| \nabla X_M(s,.)\|_{L^2}^{1/{\tilde{p}}}, \quad \forall s\geq 0.
\end{align}
Define for all $t\ge 0$
\begin{equation*}\label{eq:calJest1}
{\mathcal J(X_M)}(t) =   \int_0^t (t-s)^{-(1-\frac{1}{2{\tilde{p}}})}   \|X_M(s)\|_{L^2}^{3/2-3/(4{\tilde{p}})} \|\nabla X_M(s)\|_{L^2}^{3/(4{\tilde{p}})}\,ds.
\end{equation*}
Plugging the upper bounds~\eqref{upp_GN_1} and~\eqref{upp_GN_2} in the inequality~\eqref{Holder-J(X_M)}, and applying Young's inequality, there exists $C(T)\in(0,\infty)$ such that
\begin{align} 		
|J(X_M)(t,z)| &\leq 
C(T) \underset{(s,y)\in [0,t]\times [0,1]}\sup~|X_M (s)(y)|^{1/2} \;  {\mathcal J}(X_M)(t)
\notag \\ 
\label{upper-nonlin}& \leq 
\tfrac{1}{2} \underset{(s,y)\in [0,t]\times [0,1]}\sup~|X_M (s)(y)| 
+ \tfrac{(C(T))^2}{2} |{\mathcal J}(X_M)(t)|^{2}
, 
\end{align} 
Applying H\"older's inequality with conjugate exponents $\frac{8{\tilde{p}}}{8{\tilde{p}}-3}$ and $\frac{8{\tilde{p}}}{3}$ then yields
\begin{equation}\label{eq:calJest2}\begin{aligned}
 {\mathcal J}(X_M)(t) &\leq  \Big( \int_0^t  (t-s)^{-1 + 1/(8{\tilde{p}}-3) }
 \,ds \Big)^{1-3/(8{\tilde{p}})}  
\Big( \int_0^t \|X_M(s)\|_{L^2}^{4{\tilde{p}}-2} \|\nabla X_M(s)\|_{L^2}^2 \,ds \Big)^{3/(8{\tilde{p}})}
\\ & \leq  (8{\tilde{p}}-3)^{1-3/(8\tilde{p})}    
T^{1/(8{\tilde{p}})}\Big( \int_0^t  \|X_M\|_{L^2}^{4{\tilde{p}}-2} \|\nabla X_M(s)\|_{L^2}^2 \,ds \Big)^{3/(8{\tilde{p}})}.
\end{aligned}
\end{equation}
 The process $X_M$ solves a finite-dimensional stochastic differential equation with values in $H_M\subseteq D(A)$; therefore, $\underset{(t,z) \in [0,T]\times [0,1]}\sup~|X_M(t,z)| < \infty$ 
 almost surely. Thus, combining~\eqref{eq:sup_norm_error_split}, \eqref{eq:IVLinftybound}, \eqref{upper-nonlin}, and~\eqref{eq:calJest2}, we deduce that there exists  $C_{\tilde{p}}(T)\in (0,\infty)$ such that
\begin{align*}
\underset{(t,z) \in [0,T]\times [0,1]}\sup~ |X_M(t,z)| 
 & \leq 
 C_\alpha \| (-A)^{\alpha} X_0 \|_{L^{2}}
 + 2\underset{t \in [0,T]}\sup~ \| I_M(t) \|_{L^{\infty}}
 \\ & \quad +  C_{\tilde{p}}(T) \Big( \int_0^T \|X_M(s)\|_{L^2}^{4{\tilde{p}}-2} \|\nabla X_M(s)\|_{L^2}^2  \,ds \Big)^{3/(4{\tilde{p}})}. 
\end{align*} 
Raising both sides of the above equation to the power $p=\frac{4\tilde{p}}{3}$ and taking expectation, one obtains the moment bounds~\eqref{mom_X_infty} by employing Lemma~\ref{lem:expest}, 
Lemma~\ref{lem:reg_stochconv}  and the inequality \eqref{A.12Bis}.
 The proof of Lemma~\ref{lem:momLinfty} is thus completed.
\end{proof}

\subsection{Proof of Lemma~\ref{lem:X_sup_moments}}

\begin{proof}  
Fix $\lambda, \gamma \in (0,\frac{1}{2})$,  $M\in \N$ and observe that as 
$\lambda + \gamma < \alpha$ is a strict inequality we can assume (with a slightly larger $\lambda$ if needed)
 that $\lambda\neq \frac{1}{4}$; this allows us to apply Proposition~\ref{prop:DA_Sobolev_equiv}.
 Recall the mild representations~\eqref{eq:mildGalsolBurgers} and~\eqref{eq:mildGalsolBurgers_conv}, and  note that $X_M\in D(A)\subseteq W^{1,2}_{0}$ a.s.\ by construction. 
 Thus, we only need to prove that~\eqref{eq:XM_moment_mixed_reg} holds.
By Proposition~\ref{prop:DA_Sobolev_equiv}, Lemma~\ref{lem:A_analytic}, the  estimate~\eqref{eq:A_fracpownorminc}, and the inequalities $\| e^{tA}\|_{\mathcal{L}(L^2)} \leq 1$ and
 $\| P_M \|_{\mathcal{L}(L^2)} = 1$ one has
\begin{equation}
\begin{aligned}
&\| (-A)^{\lambda}(e^{t A}P_M X_0 - e^{sA} P_M X_0 )\|_{L^2}  
=\| (-A)^{\lambda} (e^{(t -s)A} - I) e^{sA} P_M X_0 \|_{L^2} 
\\ 
&\quad  \leq  \| (-A)^{-\gamma} (e^{(t-s)A} - I ) e^{s A} P_M (-A)^{\lambda+\gamma} X_0 \|_{L^2}
\leq(t-s)^{\gamma} \| (-A)^{\alpha} X_0 \|_{L^2},
\end{aligned}
\end{equation}
for all $0\leq s < t < \infty$ and all $M\in \N$.
This proves  
\begin{equation}\label{eq:init_value_reg}
    \underset{M\in \N}\sup~ \E \bigl[ \| t\mapsto  (-A)^{\lambda} e^{t A}P_M X_0  \|_{C^{\gamma}([0,T],L^2) }^{p} \bigr] 
    \leq  \E \bigl[ \| (-A)^{\alpha} X_0 \|_{L^2}^p\bigr].
\end{equation}
The required regularity of the stochastic convolution has been obtained in Lemma~\ref{lem:reg_stochconv}. 
Hence it only remains  to prove that the deterministic convolution in~\eqref{eq:mildGalsolBurgers} satisfies the desired regularity result. To this end, 
first of all note that by Lemma~\ref{lem:A_analytic}, 
 Lemma~\ref{lem:deriv_bdd_fracdomains}, and the identity $(-A)^{-1/2}P_M = P_M (-A)^{-1/2}$, there exists $C_{\delta,\lambda}\in (0,\infty)$ such that
\begin{align*}
\| (-A)^{\delta+\lambda} e^{tA} P_M \nabla  v \|_{L^2}
& \leq \| (-A)^{\delta+\lambda + 1/2} e^{tA} \|_{\mathcal{L}(L^2)} 
\| (-A)^{-\frac{1}{2} }\nabla \|_{\mathcal{L}(L^2)} 
\| v \|_{L^2}
 \\ &\leq C_{\delta,\lambda} t^{-(\delta + \lambda+1/2)} \| v \|_{L^2}
\end{align*}
for all $\delta\geq 0$, $t\in (0,\infty)$ and $v\in L^2$.
It then follows from~\cite[Proposition 3.6]{CoxHausenblas:2013} 
with $Y_1=Y_2=L^2$, $\Psi(t)= (-A)^{\lambda} e^{tA} P_M\nabla $, $\theta=1-\gamma$, and $g(t) = C t^{-(\lambda+\gamma+1/2)}$ that there exists $C_{\gamma,\lambda}(T)\in(0,\infty)$ such that for any $\Phi\in L^{\infty}([0,T],L^2)$ one has
\begin{align*}
\underset{M\in \N}\sup~ \left\| 
 t\mapsto(-A)^{\lambda}\int_0^{t} e^{(t-s)A}P_M \nabla \Phi(s) \,ds
\right\|_{C^{\gamma}([0,T],L^2)}
&
\leq C_{\gamma,\lambda}(T) \| \Phi \|_{L^{\infty}([0,T],L^2)}.
\end{align*}
This implies
\begin{equation}
\underset{M\in \N}\sup~ 
\E \left[ \left\| 
 t\mapsto (-A)^{\lambda} \int_0^{t} e^{(t-s)A} P_M \nabla X^2_M(s) \,ds
\right\|_{C^{\gamma}([0,T],L^2)}^{p}\right]
\leq C_{\gamma,\lambda}(T) \underset{M\in \N}\sup~ \E\bigl[\| X_M \|_{L^{\infty}([0,T],L^4)}^{2p}\bigr].
\end{equation}
Using~\eqref{eq:mildGalsolBurgers} and the identity  $P_M \nabla X_M^2 = B_M(X_M)$, we deduce that the proof reduces to check that 
\[
\underset{M\in \N}\sup~\E \bigl[\| X_M \|_{L^{\infty}([0,T],L^4)}^{2p} \bigr] \leq C \Bigl(1+\E\bigl[\| (-A)^\alpha X_0 \|_{L^2}^{2p} \bigr]+  \E\bigl[ \| X_0 \|_{L^2}^{6p}\bigr] \Bigr).
\]
Since $\|\cdot\|_{L^4}\le \|\cdot\|_{L^\infty}$, this is a straightforward consequence of Lemma~\ref{lem:momLinfty}. The proof of Lemma~\ref{lem:X_sup_moments} is thus completed. 

\end{proof}

\section{Proof of Lemma~\ref{lemma:Kolmogorov_randomIV}}\label{app:Kolmogorov}

\begin{proof}[Proof of Lemma~\ref{lemma:Kolmogorov_randomIV}]
For $n\in \N$ let $\psi_n \in C^2(\R,\R)$ be a monotone increasing function with bounded first and second derivative satisfying 
\begin{equation*}
\psi_n(x)
=
\begin{cases}
-2n, & x \in (-\infty,-2n];\\
x, & x\in [-n,n];\\
2n, & x\in [2n,\infty),
\end{cases}
\end{equation*}
and define $\varphi_n = \psi_n \circ \varphi$. The monotone convergence theorem for the conditional expectation implies that it suffices to verify~\eqref{eq:Kolmogorov_randomIV} for $\varphi_n$, i.e., we can from now on assume without loss of generality that $\varphi$ is bounded.\par 
Recall that $X_M^x$ denotes the solution to~\eqref{eq:GalsolBurgers} with initial value $x \in L^2$. If $X_0$ is a simple $\cF_0$-measurable $L^2$-valued random variable, say, $X_0 =\sum_{k=1}^{n}x_k \mathds{1}_{A_k}$ with $x_1,\ldots,x_n\in L^2$ and $A_1,\ldots,A_n\in \cF_0$, then it follows from the uniqueness of the solution (see Section~\ref{ssec:setting}) that $X_M(t) =\sum_{k=1}^{n}X^{x_k}_M(t) \mathds{1}_{A_k}$ for all $M\in \N$ and all $t\in [0,T]$. The identity~\eqref{eq:Kolmogorov_randomIV} then follows from the definition of $u_M$ (see~\eqref{eq:Kolmogorov}) and the fact that $X_M^{x_k}$ is independent of $\cF_0$ for all $k\in \{1,\ldots,n\}$.\par 
Next, recall that $(h_j)_{j\in \N}$ is an orthonormal basis for $L^2$ and define
\begin{equation}
    X_{0,n} = 
    \sum_{j=1}^{n}
    \sum_{k=-2^{-2n}}^{k=2^{2n}}
     k2^{-n} h_j 
     \mathds{1}_{\{\langle X_0, h_j\rangle_{L^2} \in   [k2^{-n}, (k+1)2^{-n}) \}},\quad n\in \N.
\end{equation}
Note that $(X_{0,n})_{n\in \N}$ is a sequence of simple $\sigma(X_0)$-measurable $L^2$-valued random variables such that $\lim_{n\rightarrow \infty} \| X_{0,n} - X_{0} \|_{L^2} =0$ a.s. Moreover, $\sigma(X_{0,n})\subseteq \sigma(X_{0,n+1})$ and $\sigma(X_0)=\sigma(\{X_{0,n}\colon n\in\N\})$. Let $X_{M,n}$ denote the solution to~\eqref{eq:GalsolBurgers} with initial value $P_M X_{0,n}$ ($M,n\in \N$). Thanks to Doob's martingale convergence theorem one has, for all $p\in (1,\infty)$ and all $\xi\in L^p(\Omega)$, that
\begin{equation}\label{eq:convcondexp}
\lim_{n\rightarrow \infty} \E[ \xi | X_{0,n}] = \E [ \xi| X_0]\quad \text{a.s. and in $L^p$}.
\end{equation}
Note that for all $n,M\in \N$ and all $t\ge 0$ one has
\begin{equation*}
\begin{aligned}
    & X_M(t) - X_{M,n}(t)
    \\ &\qquad 
    =
    P_M(X_0 - X_{0,n})
    +
    \int_{0}^{t}
        \bigl[ 
            A(X_M(t) - X_{M,n}(t))
            +
            B_M(X_M(t)) - B_M(X_{M,n}(t))
        \bigr]
    \,dt.
\end{aligned}
\end{equation*}
It follows from~\cite[Remark 3.1]{LiuRockner:2010} that there exists a constant $C\in (0,\infty)$ such that one has, for all $M\in \N$ and all $x,y\in H_M$,
\begin{equation*}
\begin{aligned}\label{eq:Burgerslocmono}
\langle 
    A(x-y) + B_M(x) - B_M(y),
    x-y
\rangle_{L^2}
&=
\langle 
    A(x-y) + B(x) - B(y),
    x-y
\rangle_{L^2} \\
& \leq 
    - \tfrac{3}{4}\| \nabla(x - y) \|_{L^2}
    +
    C \| x \|_{L^4}^4
    \| x - y \|_{L^2}^2.
\end{aligned}
\end{equation*}
This, the Poincar\'e inequality~\eqref{eq:poincare}, and the Gagliardo-Nirenberg inequality~\eqref{eq:GN} (recalling that $X_M$ takes values in $H_M\subseteq W^{1,2}_0$) imply that there exists $C\in (0,\infty)$ (possibly changing values from line to line) such that one has
\begin{align*}
&\| X_M(t) - X_{M,n}(t) \|_{L^2}^2
 = \| P_M (X_{0} - X_{0,n}) \|_{L^2}^2 
\\ & \quad 
+  2\int_0^t 
    \langle 
        X_M(s) - X_{M,n}(s)
        ,  
        A(X_M(t) - X_{M,n}(t))
        +
        B_M(X_{M}(t)) - B_M(X_{M,n}(t))
    \rangle_{L^2}\,ds
\\ & \leq 
\| X_{0} - X_{0,n}\|_{L^2}^2  +
C\int_0^{t} \| X_M(s) \|_{L^4}^4 
\| X_M(s) - X_{M,n}(s) \|_{L^2}^2 \,ds 
\\ & \leq 
\| X_{0} - X_{0,n} \|_{L^2}^2  +
C\int_0^{t}  \| X_M(s)\|_{L^2}^2 \| \nabla X_M(s) \|_{L^2}^2 
\| X_M(s) - X_{M,n}(s) \|_{L^2}^2 \,ds 
\end{align*}
for all $n,M\in \N$ and all $t\in [0,T]$.
Gronwall's lemma hence implies 
\begin{align*}
\| X_M(t) - X_{M,n}(t) \|_{L^2}^2
\leq \| X_{0} - X_{0,n} \|_{L^2}^2 \exp\left( C
\int_0^t \| X_M(s) \|_{L^2}^2 \| \nabla X_M(s) \|_{L^2}^2  \,ds \right)
\end{align*}
for all $t\in [0,T]$, $M\in \N$.
This, the fact that 
$
    \lim_{n\rightarrow \infty}
    \| X_0 - X_{0,n} \|_{L^2}=0$ a.s., 
and inequality~\eqref{pmoments} from Lemma~\ref{lem:expest} imply that 
$\lim_{n\rightarrow \infty}
\| X_M(t) - X_{M,n}(t) \|_{L^2}=0$ a.s.\ for all $t\in [0,T]$ and all $M\in \N$.
As the mapping $\varphi$ is assumed to be bounded, the dominated convergence theorem implies that 
\begin{equation}
\lim_{n\rightarrow \infty} \E [ | \E[ \varphi(X_{M,n}(t)) - \varphi(X_{M}(t))| X_{0,n} ] | ] \leq 
\lim_{n\rightarrow \infty} \E [ |  \varphi(X_{M}(t)) - \varphi(X_{M,n}(t)) | ] =0
\end{equation}
for all $t\in [0,T]$ and all $M\in \N$. Thus we can extract a subsequence $(n_k)_{k\in \N}$
such that 
\begin{equation}
\lim_{n\rightarrow \infty}   \E[ \varphi(X_{M,n_k}(t)) - \varphi(X_{M}(t))| X_{0,n} ] 
= 0 \quad \text{a.s.}
\end{equation}
This combined with~\eqref{eq:convcondexp} implies that one has
\begin{equation}
\lim_{k\rightarrow \infty} \E\big[ \varphi(X_{M,n_k}(t)) | X_{0,n_k} \big] 
= \E[ \varphi(X_{M}(t)) | X_{0} ] \quad \text{a.s.}
\end{equation}
for all $t\in [0,T]$, $M\in \N$.
This, the fact that we verified~\eqref{eq:Kolmogorov_randomIV} for simple random variables, and the fact that 
$u_M(t,\cdot)$ is continuous completes the proof of~\eqref{eq:Kolmogorov_randomIV}.
\end{proof}

\bigskip
\section*{Acknowledgements}
The authors would like to thank David Nualart and Arnulf Jentzen for helpful comments.

The work of Sonja Cox is supported by the NWO grant VI.Vidi.213.070. The work of Charles-Edouard Br\'ehier is partially supported by the SIMALIN (ANR-19-CE40-0016) project operated by the French National Research Agency. 
 Annie~Millet's research has been conducted within the FP2M federation (CNRS FR 2036).

\appendix


\end{document}